\documentclass[11pt, twoside]{amsart}
\usepackage{amsmath,mathtools}
\usepackage{amsmath, amsthm, amssymb, amsfonts, enumerate}
\usepackage[colorlinks=true,linkcolor=blue,urlcolor=blue,citecolor=blue]{hyperref}
\usepackage{dsfont}
\usepackage{color}
\usepackage{todonotes}
\usepackage{epstopdf}
\usepackage{bbm}

\usepackage[utf8]{inputenc}
\usepackage{bm}
\usepackage{amsfonts}
\usepackage{amsfonts}
\usepackage{textcomp}
\usepackage{amssymb}
\usepackage{float}
\usepackage{tikz}
\usepackage{epsfig}
\usepackage{amsmath}
\usepackage[english]{babel}
\usepackage{a4}
\usepackage{enumerate}
\usepackage{tcolorbox}
\usepackage{soul}
\usepackage{geometry}
\newcommand{\stkout}[1]{\ifmmode\text{\sout{\ensuremath{#1}}}\else\sout{#1}\fi}
\usepackage{csquotes}

\newtheorem{theorem}{Theorem}[section]

\newtheorem{remark}[theorem]{Remark}
\newtheorem{assumption}[theorem]{Assumption}
\newtheorem{lemma}[theorem]{Lemma}

\newtheorem{definition}[theorem]{Definition}

\newtheorem{example}[theorem]{Example}

\def \E{\mathsf{E}}

\def \P{\mathsf{P}}
\def \R{\mathbb{R}}

\def\d{\mathrm{d}}

\definecolor{red}{rgb}{1.0,0.0,0.0}

\definecolor{blu}{rgb}{0.0,0.0,1.0}

\definecolor{gre}{rgb}{0.03,0.50,0.03}

\usepackage[T1]{fontenc}

\usepackage{wrapfig}
\usepackage{epsf}

\title[Sensitivity  of  functionals of McKean-Vlasov SDEs]{Sensitivity of  functionals   of  McKean-Vlasov SDE's  with respect to the initial distribution}

\date{\today}
\author[de Feo]{Filippo de Feo}
\address{Filippo de Feo, Department of Economics and Finance, Luiss Guido Carli University, Rome (Italy) and Institut für Mathematik, Technische Universität Berlin, Berlin (Germany)}
\email{\href{mailto:defeo@math.tu-berlin.de}{defeo@math.tu-berlin.de}}

\author[Federico]{Salvatore Federico}
\address{Salvatore Federico, Dipartimento di Matematica, Universit\`a ``Alma Mater'' di Bologna (Italy)}
\email{\href{mailto:s.federico@unibo.it}{s.federico@unibo.it}}

\author[Gozzi]{Fausto Gozzi}
\address{Fausto Gozzi, Department of Economics and Finance, Luiss ``Guido Carli'' University, Rome (Italy)}
\email{\href{mailto:fgozzi@luiss.it}{fgozzi@luiss.it}}

\author[Touzi]{Nizar Touzi}
\address{Nizar Touzi, New York University, Tandon School of Engineering, New York (USA).}
\email{\href{mailto:nizar.touzi@nyu.edu}{nizar.touzi@nyu.edu}}

\numberwithin{equation}{section}

\begin{document}
\maketitle


\begin{abstract}
We examine the sensitivity at the origin of the distributional robust optimization problem in the context of a model generated by a mean field stochastic differential equation. We adapt the finite dimensional argument developed by Bartl, Drapeau, Obloj \& Wiesel to our framework involving the infinite dimensional gradient of the solution of the mean field SDE with respect to its initial data. We revisit the derivation of this gradient process as previously introduced by Buckdahn, Li \& Peng, and we complement the existing properties so as to satisfy the requirement of our main result. We use the theory developed in the context of a mean-field systemic risk model by  evaluating the sensitivity  with respect to the initial distribution for the variance of the log-monetary reserve of a representative bank.
\end{abstract}

\section{Introduction}

Distributionally robust optimization (DRO) has been very popular in the recent Operations Research literature. The main idea is to formulate the traditional worst case evaluation of some criterion $g(\mu)$ of a model defined by a probability measure $\mu$ outside the restricted framework of parametric models, and to consider instead all possible deviations of the underlying model in the set of probability measures. The discrepancy between such probability measures was evaluated through the Kullback divergence in Lam \cite{lam2016robust}, the total variation distance in Farokhi \cite{farokhi2023distributionally}, a criterion based a cumulative distribution functions in  Bayraktar \& Chen \cite{bayraktar2023nonparametric}, and more recently through the $p-$Wasserstein distance in Esfahani \& Kuhn  \cite{EsfahaniKuhn}  and Blanchet   \& Murthy \cite{blanchet2019quantifying}, see also Neufeld, En,  \& Zhang \cite{neufeld2024robust},
Fuhrmann, Kupper  \& Nendel \cite{fuhrmann_wasserstein_2023}, Blanchet \& Shapiro \cite{blanchet2023statistical}, Nendel \& Sgarabottolo \cite{nendel2022parametric}, Blanchet, Chen \& Zhou \cite{blanchet2022distributionally}, Blanchet,  Kang \& Murthy \cite{blanchet2019robust}, 
Blanchet, Li, Lin \& Zhang \cite{blanchet2024distributionally},
Lanzetti, Bolognani \& Dörfler \cite{lanzetti2022first},
Yue, Kuhn \& Wiesemann  \cite{yue2022linear}.

By restricting to models in the $p-$Wasserstein ball with radius $\delta$ centered at the model of interest, the DRO value function reduces to a scalar function of $\delta$ defined as the worst evaluation of the criterion $g$ in this neighborhood. Our interest here is on the sensitivity at the origin of the DRO as derived in Bartl, Drapeau, Obloj \& Wiesel \cite{bartl2021sensitivity}, see also Bartl \& Wiesel \cite{bartl2023sensitivity} for similar results in the context of the adapted Wasserstein ball, Jiang \& Obloj \cite{jiang2024sensitivity} for the continuous time extension of \cite{bartl2023sensitivity}, and Sauldubois \& Touzi \cite{sauldubois2024first} for the martingale Wasserstein ball with possibly restricted marginals.

In this paper, we consider the case where the criterion $g(\mu):=\phi\,\circ\,(T_\#\mu)$ where $T$ is a transport map of the initial law $\mu$ through the solution of a mean field stochastic differential equation. This setting is motivated by the huge interest in mean field models resulting from interacting population of agents in several application areas. The seminal papers of Lasry \& Lions \cite{LasryLions} and Huang, Malham\'e \& Caines  \cite{HuangMalhameCaines} introducing mean field games initiated a large interest interacting populations control and games problems and their mean field limit, see also Carmona \& Delarue \cite{carmona_delarue}. We are particularly interested in applications to default contagion and systemic risk modeling in financial mathematics, e.g., see Fouque, Carmona \& Sun \cite{carmona_fouque_sun}, Nadtotchiy \& Shkolnikov \cite{NadtochiyShkolnikov}, Djete \& Touzi \cite{DjeteTouzi}, and De Crescenzo, de Feo \& Pham \cite{decrescenzo_defeo_pham}.

The current setting in this paper with criterion involving the law of a mean field SDE is in the general spirit of the static setting considered by Bartl, Drapeau, Obloj \& Wiesel \cite{bartl2021sensitivity}. However, the criterion considered in this paper $g(\mu):=\phi\,\circ\,(T_\#\mu)$ does not fall in the setting considered in there, and involves a nontrivial nonlinearity due to the mean field dependence.
In order
to adapt their derivation approach of the DRO sensitivity at the origin and to characterize it in our dynamic setting, we essentially analyze  the regularity of this transport map.  This requires a careful study of the gradient process of the solution of the mean field stochastic differential equation with respect to its initial condition which was introduced previously by  Buckdahn, Li, Peng \& Rainer  \cite{buckdahn_li_peng_rainer} (see also Chassagneux, Crisan,  \& Delarue  \cite{chassagneux_delarue_crisan}). We revisit the approach of \cite{buckdahn_li_peng_rainer} and complement their result as follows. Under appropriate conditions, this gradient process is the Gateaux derivative of the solution of the mean field SDE, is a uniformly strongly continuous functional of its initial condition and is bounded in the operator norm. The corresponding adjoint process inherits the last properties and is the main ingredient for the expression of the DRO sensitivity at the origin. We use the theory developed in the context of a mean-field systemic risk model \cite[Section 2]{carmona_fouque_sun} by  evaluating the sensitivity  with respect to the initial distribution for the variance of the log-monetary reserve of a representative bank.

The paper is organized as follows. Section \ref{sec:framework} introduces the general framework. Section \ref{sec:main} contains the main results together with the justification of the distributionally robust sensitivity. All arguments related to the differentiability of the solution of the McKean-Vlasov SDE with respect to its initial law are reported in Section \ref{sec:diff}. Finally, we provide in Section \ref{sec:example} an application in the context of the systemic risk model of Carmona, Fouque \& Sun \cite{carmona_fouque_sun}.

\section{The underlying framework}
\label{sec:framework}

\subsection{Basic notations.}\label{subsection:notations} We denote by $\mathbb R^{d \times h}$ the space of $d \times h-$matrices with entries in $\mathbb R$ and  by $|\cdot|$ the Frobenius norm on this space. The transpose of a matrix is denoted by the superscript $^{T}$. 

Given a Polish space $E$, we denote by  $\mathcal{B}(E)$ the Borel $\sigma-$algebra on $E$ and by $\mathcal{P}(E)$ the set of  probability measures on $(E, \mathcal{B}(E)).$  The Dirac's delta  measure concentrated at  $x\in E$ is denoted by  $\delta_x$. We  denote the set of couplings in $\mathcal{P}(\mathbb R^{d } \times \mathbb R^{d })$ with given marginals $\mu, \mu' \in \mathcal{P}(\mathbb R^{d})$ by
$\Pi(\mu,\mu'):=\big\{\pi \in \mathcal{P}(\mathbb R^{d } \times \mathbb R^{d }):  \ \pi(\cdot \times \mathbb R^{d })=\mu,  \ \pi(\mathbb R^{d } \times \cdot)=\mu^{\prime}\big\} .$

Given a  probability space $(\Omega, \mathcal F, \P)$ and a random variable $X \colon \Omega \to \mathbb R^{d \times h}$, we denote by $\P_X$ the law of $X$. For $p>0$ we denote by $L^p(\mathcal F; \mathbb R^{d \times h})$ the Banach space of $p-$integrable random variables $X$ with norm
$\| X\|_{L^p}:=(\E |X|^p)^{1/p}$. We denote $L^p(\mathcal F):=L^p(\mathcal F; \mathbb R)$.
When $p=2$ we denote by $\langle \cdot,\cdot \rangle_{L^2}$ the corresponding inner product on this space. 
A probability space $(\Omega, \mathcal F, \P)$  is said to be atomless if for every $A \in \mathcal F$ such that $\P(A)>0$, there exists  $B \subsetneq A$ such that $\P (B)>0.$ It is well known  \cite{carmona_delarue} that, in this case,  for every $\mu \in \mathcal P(\mathbb R^{d \times h})$ there exists a random variable $X \colon \Omega \to \mathbb R^{d \times h}$ random variable such that $\P_X=\mu.$ 

\subsection{The Wasserstein space.}
We define 
$\mathcal{P}_2(\mathbb R^{d }):=\left\{\mu \in \mathcal{P}(\mathbb R^{d }):  \|\mu \|_2<\infty\right\},$
where 
$\|\mu \|_2:=\left ( \int_{\mathbb R^{d}} | x|^2 \mu(d x) \right)^{1/2}.$   We endow $\mathcal{P}_2(\mathbb R^{d })$ with the $2$-Wasserstein distance 
$$
\mathcal{W}_2\left(\mu, \mu^{\prime}\right):=\inf_{\pi \in \Pi(\mu,\mu')} \left( \int_{\mathbb R^{d } \times\mathbb R^{d }} |x-y|^2 \pi(d x, d y) \right)^{\frac{1}{2}},\quad \mu, \mu^{\prime} \in \mathcal{P}_2(\mathbb R^{d }).
$$
The space $\left(\mathcal{P}_2(\mathbb R^{d }), \mathcal{W}_2\right)$ is a Polish space.  It is well-known \cite{carmona_delarue} that an optimal coupling  always exists, i.e. there exists
 $\pi^* \in \Pi(\mu,\mu')$ such that $\mathcal{W}_2\left(\mu, \mu^{\prime}\right)=\int_{\mathbb R^{d } \times\mathbb R^{d }} |x-y|^2 \pi^*(d x, d y)$.  We recall \cite{carmona_delarue} that the following equality holds:
 \begin{equation}\label{eq:W2(mu,delta0)}
 \|\mu\|_2=\mathcal W_2(\mu,\delta_0),  \quad  \mu \in \mathcal{P}_2(\mathbb R^{d }).
 \end{equation}
 We denote 
$B^2_\delta (\mu):=\left \{ \mu' \in \mathcal{P}_2(\mathbb R^{d }) : \ \  \mathcal{W}_2\left(\mu', \mu\right)  \leq  \delta \right  \}$.

If   $(\Omega, \mathcal F, \P)$ is atomless, we have
$
\mathcal{P}_2(\mathbb R^{d })=\left\{ \P_X: X \in L^2(\mathcal F; \mathbb R^{d } )\right\},
$
that is, given $\mu\in \mathcal{P}_2(\mathbb R^{d })$, we may construct $X\in L^2(\mathcal F;\mathbb R^d)$ such that $\P_{X}=\mu$.
Hence, given $ \mu, \mu' \in \mathcal{P}_2(\mathbb R^{d })$, we have $\Pi(\mu,\mu') \subset \mathcal{P}_2(\mathbb R^{d } \times \mathbb R^{d })$ and
$
\Pi(\mu,\mu')=\Big\{\P_{(X,X')}  : \ X,X' \in  L^2(\mathcal F; \mathbb R^{d } ), \ \mathbb P_X = \mu, \mathbb P_{X'} = \mu' \Big\} .
$
Thus,  for every  $\mu, \mu^{\prime} \in \mathcal{P}_2(\mathbb R^{d })$, we have
$
\mathcal{W}_2\left(\mu,\mu^{\prime}\right):=\inf \big\{  \| X-X'\|_{L^2}: \ X,X' \in L^2(\mathcal F; \mathbb R^{d }),\ X \sim \mu, \ X' \sim \mu' \big\}.
$
Finally, there exist
 $ X,X' \in L^2(\mathcal F; \mathbb R^{d })$ such that $\P_{(X , X')}=\pi^* \ \ \mbox{and} \ \ \mathcal{W}_2\left(\mu, \mu^{\prime}\right)=\| X-X'\|_{L^2}.$
 
 \subsection{Derivatives in the Wasserstein space.} \label{subsec:derivatives_wasserstein}
 We recall here the definition of linear functional derivative and of $L-$derivative \cite{carmona_delarue}. 
 \begin{definition}\label{def:L_derivative} 
{\rm (i)} A map $\phi \colon  \mathcal{P}_2\left(\mathbb{R}^d\right) \rightarrow \mathbb R$ is said to have  a  \emph{linear functional derivative} if there exists a map $\delta_\mu \phi: \mathcal{P}_2\left(\mathbb{R}^d\right) \times \mathbb{R}^d \rightarrow \mathbb R$ continuous for the product topology, such that the function $ x \mapsto \delta_\mu \phi(\mu, x)$ has at most  quadratic growth in $x$ locally uniformly in $\mu$ and
$$
\phi (\mu^{\prime})-\phi(\mu)=\int_0^1\int_{\mathbb{R}^d} \delta_\mu \phi (\mu^\lambda,x) \left(\mu^{\prime}-\mu\right)(dx) d \lambda, \ \ \ \mu, \mu^{\prime}  \in \mathcal{P}_2 (\mathbb{R}^d),
$$
where  $\mu^\lambda:=\lambda ( \mu^{\prime}-\mu)+ \mu$. 
We  denote by  $\mathcal C^1_{LF}( \mathcal{P}_2(\mathbb{R}^d);\R)$ the class of such  maps.
\smallskip

{\rm (ii)} A map  $\phi \colon  \mathcal{P}_2\left(\mathbb{R}^d\right) \rightarrow \mathbb R$ is said to be \emph{$L$-differentiable} if: it   belongs to $\mathcal{C}_{LF}^1\left( \mathcal{P}_2\left(\mathbb{R}^d \right);  \mathbb{R} \right)$; 
 for every $\mu  \in \mathcal{P}_2\left(\mathbb{R}^d\right)$ the real-valued  map $x \mapsto\delta_\mu \phi(\mu  , x)$ is  differentiable;
 the map $(\mu, x) \mapsto \partial_x \delta_\mu \phi(\mu, x) \in \mathbb{R}^{d}$ is  continuous and has  at most of quadratic growth in $x$, locally uniformly in $\mu$. In this case,
   the map $ \partial_x \delta_\mu \phi $ is called the \emph{$L$-derivative} of $\phi$. We denote the set of such maps by $C^1\left(\mathcal{P}_2\left(\mathbb{R}^d\right);\mathbb{R} \right)$.   Analogously, we define the space $C^1\left(\mathcal{P}_2\left(\mathbb{R}^d\right);\mathbb{R}^{d\times m} \right)$.  
\end{definition}
We recall  that the $L-$derivative $\partial_x \delta_\mu \phi$
coincides with the Lions derivative \cite{carmona_delarue} {and the Wasserstein gradient \cite{villani2009optimal}}.
We extend the previous definition as follows.
{\begin{definition}
\begin{enumerate}[(i)]
\item[]
\item $C^{1} (\mathbb R^d \!\times\! \mathcal{P}_2\left(\mathbb{R}^d\right);\mathbb R^{d \times m}  )$ is the space of maps $\phi \colon \mathbb R^d \times \mathcal{P}_2\left(\mathbb{R}^d\right) \longrightarrow \mathbb R^{d \times m}$ such that
	\begin{itemize}
	\item $\phi(\cdot ,\mu) \!\in\! C^1\left(\mathbb{R}^d;\mathbb{R}^{d\times m} \right)$ for all $\mu  \in \mathcal{P}_2\left(\mathbb{R}^d\right)$ with partial Jacobian denoted by $\partial_{x} \phi (x,\mu)$,
\item   $\phi(x,\cdot) \!\in\! C^1\left(\mathcal{P}_2\left(\mathbb{R}^d\right);\mathbb{R}^{d\times m} \right)$ for all $x \in \mathbb R^d$, with partial $L-$derivative denoted by $\partial_{\tilde x} \delta_{\mu} \phi(x,\mu,\tilde x)$. 
\end{itemize}
\smallskip
\item 
We denote by 
\begin{itemize}
\item  $C^{1}_b (\mathbb R^d \times \mathcal{P}_2\left(\mathbb{R}^d\right);\mathbb R^{d \times m}  )$ the subspace of functions $\phi$ with bounded $\partial_{x} \phi$ and $\partial_{\tilde x} \delta_{\mu} \phi$.
\item $C^{1,1} (\mathbb R^d \times \mathcal{P}_2\left(\mathbb{R}^d\right);\mathbb R^{d \times m}  )$ (resp., $C^{1,1}_b (\mathbb R^d \times \mathcal{P}_2\left(\mathbb{R}^d\right);\mathbb R^{d \times m}  )$) the corresponding subspace of functions $\phi$ with Lipschitz (resp., Lipschitz and bounded) partial gradients $\partial_{x} \phi$ and $\partial_{\tilde x} \delta_{\mu} \phi$.
\end{itemize}
\end{enumerate}
\end{definition}}

\subsection{McKean-Vlasov stochastic differential equation} Let $T>0$ and let $\left \{B_t\right \}_{t \in [0,T]}$ be an $m$-dimensional Brownian motion defined on a complete atomless probability space $(\Omega, \mathcal{F}, \P)$. We denote by $\mathbb{F}=(\mathcal{F}_t)_{t \in [0,T]}$ the completion of the filtration generated by $B$, which is also right-continuous, so that it satisfies the usual conditions. We  denote by $\mathcal S^2([0,T];\mathbb R^{d \times h})$ the space of continuous $\mathbb{F}-$adapted processes  $Y$ with values in $\mathbb R^{d \times h}$  such that $\|Y \|_{\mathcal S^2}^2:=\E\left[ \sup_{t \in [0,T]} |Y_t|^2\right] < \infty.$

{Given  $\xi \in L^2\left(\mathcal F_0; \mathbb R^d\right)$ and $b=(b^0,b^1) \colon \mathbb R^d \times \mathcal P(\mathbb R^d) \to \mathbb R^d\times\mathbb R^{d \times m}$, we  consider a SDE of  McKean-Vlasov (MKV) type: 
\begin{equation}\label{eq:sde}
X_t^{\xi}=\xi+\int_0^t b( X_s^{\xi}, \P_{X_s^{\xi}} ) d Z_s, \quad t \in[0, T],
~\mbox{where}~
Z_t
=\begin{bmatrix}
Z_t^0\\
Z_t^1
\end{bmatrix}
:=\begin{bmatrix}
t\\
B_t
\end{bmatrix}.
\end{equation}
Under the standard Lipschitz condition
\begin{equation}\label{ass:sde}
\left|b( x, \mu)-b\left(x^{\prime}, \mu^{\prime}\right)\right|  \leq C \left(\left|x-x^{\prime}\right|+\mathcal{W}_2\left(\mu, \mu^{\prime}\right)\right),
~(x, \mu),\left(x^{\prime}, \mu^{\prime}\right) \in \mathbb{R}^d \times \mathcal{P}_2\big(\mathbb{R}^d\big),
\end{equation}
for some constant $C>0$, there exists a unique  solution $X^{\xi}$ to \eqref{eq:sde} in the class of processes   $\mathcal S^2([0,T];\mathbb R^d)$, see e.g. \cite[Th.\,4.21 and Lemma 4.34]{carmona_delarue}. Moreover, there exists  $C>0$ such that  
\begin{equation}\label{eq:estimate_E_X_X0_wrt_initial_cond}
\|X^{\xi}\|_{\mathcal S^2} \leq C \|\xi\|_{L^2}
~\mbox{and}~
\|X^{\xi}-X^{ \xi'}\|_{\mathcal S^2} \leq C \|\xi- \xi'\|_{L^2},
~\mbox{for all}~
\xi, \xi' \in L^2\big(\mathcal F_0 ; \mathbb{R}^d\big).
\end{equation}
By the uniqueness in law of the solution to the MKV SDE, it follows that the law of the process $X^\xi$ is independent of the choice of the initial r.v. $\xi$ in the set
\begin{equation}\label{class}
\mathcal{R}_{\mu_{0}}:=
\big\{\xi\in L^2(\mathcal F_0;\mathbb R^d): \ \P_{\xi}=\mu_{0}\big\}.
\end{equation}
In the following, we shall often confuse $\mu_{0}$ with the reference initial conditions $\xi\in\mathcal{R}_{\mu_{0}}$, and abuse notation writing $\P_{\!X_t^{\mu_0}}$ instead of $\P_{\!X_t^{\xi}}$ for an arbitrary $\xi\in\mathcal{R}_{\mu_{0}}$.
}

\section{Sensitivity of functionals for McKean-Vlasov SDEs}\label{sec:sensitivity} 
\label{sec:main} 
    
Let $\phi \colon  \mathcal{P}_2(\mathbb{R}^d) \longrightarrow \mathbb R$, our  goal is to analyze the differentiability at the origin $r=0$ of the map
$$
\Phi(\mu_0 ,r) := \sup_{\mu_0' \in B^2_r(\mu_0)} \phi \big(\mu_T'\big), \ \ \   \quad   \mu_0 \in \mathcal{P}_2(\mathbb{R}^d), \ r \geq 0 ,
$$
where {$\mu_t^\prime:=\P_{\!X_t^{\mu'_0}}=\P_{\!X_t^{\xi'}}$ for $t\in[0,T]$, and
 $X^{\xi'}$ is the solution to \eqref{eq:sde} with an arbitrary initial condition $\xi'\in\mathcal{R}_{\mu_{0}'}$.
 When $\mu_0'=\mu_0,$ we denote the corresponding reference starting condition by $\xi$ and  set $\mu_t:=\P_{X_t^{\xi}}$ for $t\in[0,T]$. }
 
 \begin{assumption}\label{ass:linear_fuunctional_derivative} The functional $\phi$ satisfies the following.
\begin{enumerate}[(i)]
\item There exists $C>0$ such that
$| \phi (\mu )|\leq C(1+\|\mu \|_2^2),$ for all  $\mu \in \mathcal P_2(\mathbb R^d).$ 
\item  $\phi \in C^1\left(\mathcal{P}_2\left(\mathbb{R}^d\right)\right)$, there exists a modulus of continuity $\varpi$  such that
\begin{align*}
& |\partial_x \delta_\mu \phi (\mu',x )- \partial_x \delta_\mu \phi (\mu,x )| \leq \varpi (\mathcal W_2(\mu',\mu)), & \mu, \mu' \in \mathcal P_2(\mathbb R^d),x \in \mathbb R^d,
\end{align*}
and for all $\mu \in \mathcal P_2(\mathbb R^d)$ a modulus of continuity $\varpi_\mu$ such that 
\begin{align*}
& |\partial_x \delta_\mu \phi (\mu,x' )- \partial_x \delta_\mu \phi (\mu,x )| \leq  \varpi_\mu (|x'-x|), &   x,x' \in \mathbb R^d.
\end{align*}
\end{enumerate}
\end{assumption}
We provide some examples of functionals satisfying the previous assumption and which justify that the variance criterion considered in the example of Section \ref{sec:example}  is included in our setting.
\begin{example}\label{ex:linear_functionals}
   Let $f \in C^1(\mathbb R^d)$ be a function with sub-quadratic  growth and uniformly continuous gradient $\nabla f$. Then $\phi(\mu):=\int_{\mathbb R^d} f(x) \mu(dx)$ satisfies Assumption \ref{ass:linear_fuunctional_derivative}.  Indeed, (i) follows from the subquadratic  growth of $f$; (ii) follows from the fact that  $\partial_x \delta_\mu \phi (\mu,x )=\nabla f(x)$.
\end{example}
\begin{example}\label{ex:functions_of_linear_functionals}
Let $\psi \in C^1(\mathbb R)$ have  uniformly continuous  derivative $\psi'$ and let $f\in C^1_b(\mathbb R^d)$  have  uniformly continuous gradient $\nabla f$. Then $\phi(\mu):=\psi \left (\int_{\mathbb R^d} f(x) \mu(dx)\right)$ satisfies Assumption \ref{ass:linear_fuunctional_derivative}.  Indeed, (i) follows from sublinear growth of $\psi,f$. As for (ii), we have 
$\partial_x \delta_\mu \phi (\mu,x )=\psi'\left(\int_{\mathbb R^d} f(y) \mu(dy) \right)\nabla f(x)$, so that for some modulus $\varpi$ (which may change from line to line), denoting  by $\mathcal W_1$ the $1$-Wasserstein distance, we have
\begin{align*}
   |\partial_x \delta_\mu \phi (\mu',x )-\partial_x \delta_\mu \phi (\mu,x )|
   &\leq \left|\psi'\left(\int f d\mu'\right) - \psi'\left(\int f d\mu\right) \right| |\nabla f(x)| 
   \\
   & \leq \varpi\left (\left |\int f d\mu'-\int f d\mu\right | \right)\leq \varpi\left (\mathcal W_1(\mu',\mu) \right) \leq \varpi\left (\mathcal W_2(\mu',\mu) \right),
\end{align*}
where in the third inequality we have used the Kantorovich-Rubinstein duality \cite[Corollary 5.4]{carmona_delarue} up to rescaling by the Lipschitz constant of $f$. Finally, we have
\begin{align*}
   |\partial_x \delta_\mu \phi (\mu,x' )-\partial_x \delta_\mu \phi (\mu,x )|\leq  \left| \psi'\left(\int f d\mu \right) \right| |\nabla f(x')-\nabla f(x)| 
   \leq \left| \psi'\left(\int fd\mu\right) \right| \varpi(|x'-x|). 
\end{align*}
\end{example}
\medskip
For later use, we start by providing some estimates.
\begin{lemma}\label{lemma:growth_bounds_phi_der_phi}
Let Condition \eqref{ass:sde} and Assumption \ref{ass:linear_fuunctional_derivative} (i)  hold.  Then:
\begin{enumerate}[{\rm (i)}]
\item  There exists $C>0$ such that
$| {\Phi} (\mu_0,r )|\leq  C (1+r^2+ \|\mu_0 \|_2^2),$ for every  $\mu_0 \in \mathcal{P}_2(\mathbb{R}^d), \   r >0.$
\item {Under the additional Assumption   \ref{ass:linear_fuunctional_derivative} (ii), we have: 
$$
K_R
:=
\sup_{t \in[0,T]}\sup_{\mu_0\in B^2_R(\mu_0)}
\left \| \partial_x \delta_\mu \phi (\mu_t,X_t^{\xi})\right \|_{L^2} 
<\infty
~\mbox{for all}~R>0.
$$}
\end{enumerate}
\end{lemma}

\begin{proof}
(i) For $r >0$, let $\mu_0 \in \mathcal{P}_2(\mathbb{R}^d)$ and $\mu'_0\in B^2_r(\mu_0)$. By Assumption \ref{ass:linear_fuunctional_derivative} we have for some constant $C$ that may vary from line to line in the inequalities,
$
\big|\phi \big(\mu_T'\big)\big| \leq C\big(1+\big\|\mu_T'\big \|_2^2\big)=C\big(1+\big\|X_T^{\xi'}\big\|_{L^2}^2\big)
$. Then, it follows from \eqref{eq:estimate_E_X_X0_wrt_initial_cond} together with \eqref{eq:W2(mu,delta0)} and the triangle inequality for $\mathcal W_2(\cdot,\cdot)$ that:
$$
\big|\phi \big(\mu_T'\big)\big| 
\leq
C\big(1+\big\|\xi'\big\|_{L^2}^2\big) 
=
C\big(1+\big\|\mu_0'\big\|_{2}^2\big)
\leq
C\big(1+\mathcal W_2(\mu_0',\mu_0)^2+\big\|\mu_0\big\|_{2}^2\big) 
\leq 
C (1+r^2+ \|\mu_0 \|_2^2),
$$
which induces the required claim.

 (ii)
For arbitrary $t \in[0,T]$ and $\mu_0 \in \mathcal{P}_2\left(\mathbb{R}^d\right)$ such that $\|\mu_0 \|_2 \leq  R$, {by \eqref{eq:estimate_E_X_X0_wrt_initial_cond} we have 
  $ \|\mu_t\|^{2}_{2}= \|X_t^{\xi} \|^{2}_{L^2}\leq  {C} (1+ \|\xi\|^{2}_{L^2})={C} (1+ \|\mu_0\|^{2}_{2})\leq {C} (1+R^{2})$ for some $C>0$.}
We can  then use 
\eqref{eq:estimate_E_X_X0_wrt_initial_cond} 
and  the growth bound for $|\partial_x \delta_\mu \phi |$ in Definition \ref{def:L_derivative} to get, {for some $C_{R}>0$,}
\begin{align*}
\left \| \partial_x \delta_\mu \phi (\mu_t,X_t^{\xi})\right \|_{L^2}^2 \leq  C_R (1+ \|X_t^{\xi} \|_{L^2}^2) \leq  C_R (1+ \|\xi\|_{L^2}^2)=C_R (1+ \|\mu_0\|_{2}^2)\leq {C_{R}(1+{R}^{2})}.\vspace{-.6cm}
\end{align*}
\end{proof}

To go towards our main result concerning the study of the sensitivity of the functional $\Phi$ with respect to $\mu$, we first investigate the dependence of the solution of the MKV equation with respect to the initial datum. Questions of this type  have been addressed  
in \cite{chassagneux_delarue_crisan} and in \cite{buckdahn_li_peng_rainer} (see also \cite{hao_li} for extensions to McKean Vlasov SDEs with jumps and  \cite{buckdahn_li_xing} for mean-field backward doubly stochastic differential equations). {We notice that these papers address the smoothness of the map $\xi\mapsto X_t^{\xi}$ with different approaches. While \cite{chassagneux_delarue_crisan} show the Gateaux differentiability by analyzing the dependence of the Picard iterations on the initial law, we shall follow the approach of \cite{buckdahn_li_peng_rainer} who introduced an auxiliary process with initial law concentrated at a point thus reducing the initial value of the process to a deterministic object.} 

Given two Banach spaces $(E,\|\cdot \|_E),(F,\| \cdot \|_F)$, we denote by $\mathcal L(E,F)$ the Banach space of linear bounded  operators $L$ from $E$ to $F$, endowed with the operator norm 
$|L|_{\mathcal L } := \sup_{\|x\|_E=1}\|Lx\|_F$.
The notion of derivative of the map  $\xi\mapsto X_t^{\xi}$ is provided by Theorem \ref{th:derivative_XtX0_wrt_X0} below. Such a derivative is seen as an operator  $D_{\xi}X_t^{\xi}\in\mathcal{L}\left (L^2(\mathcal F_0;\mathbb R^d); L^2(\mathcal F_t;\mathbb R^d)\right)$. Relying on the results of \cite{buckdahn_li_peng_rainer}, we are able to prove Gateaux differentiability of this map and some  estimates concerning $D_{\xi}X_t^{\xi}$ and its adjoint $(D_{\xi}X_t^{\xi})^*$.
In  Subsection  \ref{sec:derivativeBuck}, we will also  derive their explicit expressions \eqref{GD} and \eqref{eq:D_X0_XtX0*}. We  refer to Remark \ref{rem:comparison_buck} for a discussion. {Our analysis requires  to strengthen condition \eqref{ass:sde} as follows.}
\begin{assumption}\label{ass:derivatives_coeff} $b\in C^{1,1}_b (\mathbb R^d \times \mathcal{P}_2\left(\mathbb{R}^d \right) ; \mathbb R^{d} \times  \mathbb R^{d \times m})$.
 \end{assumption}
   \begin{example}\label{ex:example_coeff}
Assume $d=m=1$ for simplicity.  Let $\hat b\in C^{1,1}_b (\mathbb R \times \mathbb R ; \mathbb R \times  \mathbb R)$ and $g \in C^{1,1}_b(\mathbb R)$. Then $b(x,\mu):=\hat b \left (x,\int_{\mathbb R^d} g(y)\mu(dy)\right )$ satisfies Assumption \ref{ass:derivatives_coeff}. Indeed, denoting by $\partial_x \hat b, \partial_y \hat b$, respectively, the partial derivatives with respect to the first and second variable of $\hat b$, we have $$\partial_{x}  b (x,\mu )=\partial_{x}\hat b \left(x,\int_{\mathbb R^d} g(z) \mu(dz) \right), \ \  
 \partial_{\tilde x}  \delta_\mu b (x,\mu,\tilde x )=\partial_{y}\hat b\left(x,\int_{\mathbb R^d} g(z) \mu(dz) \right)g'(\tilde x).$$ Proceeding as in Example \ref{ex:functions_of_linear_functionals}, we have the claim.
 \end{example}
\begin{theorem}\label{th:derivative_XtX0_wrt_X0}
Let Assumption   \ref{ass:derivatives_coeff} hold. 
\\
{\rm (i)} For all  $t \in [0,T]$, the map
$
\xi\in L^2(\mathcal F_0;\mathbb R^{d})  \mapsto  X_t^{\xi}\in L^2(\mathcal F_t;\mathbb R^{d })$
 is  Gateaux differentiable with bounded differential; i.e., for all ${\xi} \in L^2 (\mathcal F_0;\mathbb R^d)$, there exists $D_{\xi} X_t^{\xi} \in \mathcal L\left (L^2(\mathcal F_0;\mathbb R^d); L^2(\mathcal F_t;\mathbb R^d)\right)$ such that
$$
\sup_{\xi \in L^2 (\mathcal F_0;\mathbb R^d)}
\sup_{t \in [0,T]}
\left \| D_{\xi}X_t^{\xi}  \right \|_{\mathcal L}<\infty
~~\mbox{and}~~
\lim_{r \to 0}\frac 1 {|r|}   \left  \|  X_t^{ \xi+r \eta}-X_t^{\xi}-r\, D_{\xi} X_{t}^{\xi} \eta \right\|_{L^2}=0,
\ \ \ \eta \in L^2 (\mathcal F_0;\mathbb R^d).
$$
Moreover, $\xi\mapsto D_{\xi}X_t^{\xi}$ is uniformly strongly continuous; that is, for every $\eta \in L^2 (\mathcal F_0;\mathbb R^d)$ there exists $ \varpi_\eta$ modulus of continuity such that  
\begin{align}
\left \|\left (D_{\xi}X_t^{\xi'} - D_{\xi}X_t^{\xi}\right) \eta\right \|_{L^2} \leq \varpi_\eta (\|\xi'-\xi \|_{L^2}), \ \  \xi, \xi' \in L^2 (\mathcal F_0;\mathbb R^d).
\end{align}
{\rm (ii)} The adjoint  $\left (D_{\xi}X_t^{\xi} \right )^* \in \mathcal L  (L^2(\mathcal F_t;\mathbb R^d); L^2(\mathcal F_0;\mathbb R^d) )$ inherits the boundedness and the uniform strong continuity, i.e. $\big\|\big(D_{\xi}X_t^{\xi}\big)^*\big\|_{\mathcal L}$ bounded uniformly in $(t,\xi)$ and for all $t \in [0,T]$:
$$
\left \|\left [ \left(D_{\xi}X_t^{\xi'}\right)^*-\left( D_{\xi}X_t^{\xi}\right)^*\right]\eta\right \|_{L^2} \leq \varpi_{\eta} (\|\xi'-\xi \|_{L^2}), \    \xi,\xi'\in L^2 (\mathcal F_0;\mathbb R^d),~\eta \in L^2 (\mathcal F_t;\mathbb R^d).
$$
for some modulus of continuity $\varpi_\eta$. 
\end{theorem}

\begin{proof} See  Subsection  \ref{sec:derivativeBuck}.
\end{proof}
Our main result is the following.

\begin{theorem}\label{th:derivative_varphi}
Let Assumptions \ref{ass:linear_fuunctional_derivative} and  \ref{ass:derivatives_coeff} hold. Then, for each  $\mu_0 \in \mathcal{P}_2\left(\mathbb{R}^d\right)$, the function $r \longmapsto {\Phi}(\mu_0,r)$ is differentiable at   $0$ and 
\begin{align*}
\frac{ \partial {\Phi}}{\partial r} (\mu_0,0 ) =\Big\|  \big(D_{\xi}X_T^{\xi} \big)^*  \partial_x \delta_\mu \phi (\mu_T,X_T^{\xi}  )\Big\|_{L^2}.
\end{align*}
\end{theorem}
{\begin{proof}
Denote $\Delta_r
:=
\frac{{\Phi}(\mu_0, r)-{\Phi}(\mu_0, 0)}{r},$ and $
\zeta:=\big(D_{\xi}X_T^{\xi} \big)^*  \partial_x \delta_\mu \phi (\mu_T,X_T^{\xi})
$ for an arbitrary $\xi\in \mathcal{R}_{\mu_{0}}$. We organize the proof in two steps. \smallskip

\noindent
\emph{Step 1.}  We first prove 
that $\limsup_{r \searrow 0} \Delta_r
\leq  \|  \zeta \|_{L^2}$.
By  definition of linear functional derivative and $L-$derivative, we may write
\begin{align*}
r\Delta_r
=  \sup_{\mu_0' \in B^2_r(\mu_0)} \left\{\phi \left (\mu_T'\right )- \phi \left (\mu_T  \right )\right\}
&=\sup_{\mu_0' \in B^2_r(\mu_0)} 
     \int_0^1\int_{\mathbb R^d} \delta_\mu \phi (\mu_T^{\lambda_1},x) (\mu_T'-\mu_T)(dx)
                 d \lambda_1
\\
&=\sup_{\mu_0' \in B^2_r(\mu_0)} 
     \int_0^1   \E \left[  \delta_\mu \phi (\mu_T^{\lambda_1},X_T^{\xi'}) 
                                  -\delta_\mu \phi (\mu_T^{\lambda_1},X_T^{\xi})   \right] d \lambda_1
\\
&= \sup_{\mu_0' \in B^2_r(\mu_0)} 
      \int_0^1\int_0^1  \E \left[  \left\langle \partial_x \delta_\mu \phi (\mu_T^{\lambda_1},X_T^{\lambda_2} ),
                                                                X_T^{\xi'}-X_T^{\xi}  
                                              \right\rangle 
                                     \right] d\lambda_2d \lambda_1,
\end{align*}
where 
$\mu_t^{\lambda_1}:=\lambda_1 ( \mu_t'-\mu_t) + \mu_t, \ \ X_t^{\lambda_2}:=\lambda_2(X_t^{\xi'}-X_t^{\xi}) +X_t^{\xi},$ $\lambda_1, \lambda_2\in [0,1], t \in [0,T].$
By Theorem \ref{th:derivative_XtX0_wrt_X0}, we  may apply \cite[Theorem 4.A, p. 148]{zeidler} to obtain the equality 
$X_T^{\xi'}-X_T^{\xi}=\int_0^1 \big[D_{\xi}X_T^{\xi^{\lambda_3}}\big](\xi'-\xi)d \lambda_3$ with $\xi^{\lambda_3}:=\lambda_3(\xi'-\xi) +\xi.$
Then, setting $\lambda=(\lambda_1, \lambda_2,\lambda_3)$,
\begin{align}
r\Delta_r
&=\sup_{\mu_0' \in B^2_r(\mu_0)} 
      \int_{[0,1]^3}  \E \left[  \left\langle \partial_x \delta_\mu \phi (\mu_T^{\lambda_1},X_T^{\lambda_2} ),
                                                           D_{\xi}X_T^{\xi^{\lambda_3}} (\xi'-\xi)
                                         \right\rangle
                                \right]d\lambda 
\nonumber \\
&=\sup_{\mu_0' \in B^2_r(\mu_0)} 
      \int_{[0,1]^3}  \E \left[  \left\langle \zeta^\lambda,
                                                           \xi'-\xi
                                         \right\rangle
                                \right]d\lambda,
~~\mbox{with}~~
\zeta^\lambda
:=
\big(D_{\xi}X_T^{\xi^{\lambda_3}}\big)^*
\partial_x \delta_\mu \phi (\mu_T^{\lambda_1},X_T^{\lambda_2} ). \label{eq:delta_Delta}
\end{align}
By the H\"older inequality, and choosing $\xi',\xi\in L^2(\mathcal F_0;\mathbb R^d)$ such that $\P_{(\xi',\xi)} \in \Pi(\mu_0,\mu'_0)$  is an  optimal coupling for $\mu'_0=\P_{\xi'},$ $\mu_0=\P_{\xi}$, this implies that
\begin{align}
\Delta_r
\leq \frac1{r} 
          \sup_{\mu_0' \in B^2_r(\mu_0)} \int_{[0,1]^3}  \big\| \zeta^\lambda\big\|_{L^2} \big\|  \xi'-\xi \big\|_{L^2}   d\lambda 
&= \frac1{r}
       \sup_{\mu_0' \in B^2_r(\mu_0)} \int_{[0,1]^3}  \big\| \zeta^\lambda\big\|_{L^2} \mathcal W_2(\mu_0',\mu_0)\   d\lambda
\nonumber\\
&\leq \sup_{\mu_0' \in B^2_r(\mu_0)} \int_{[0,1]^3}  \big\|  \zeta^\lambda\big\|_{L^2} \  d\lambda.\label{eq:equality_incremental limit_varphi}
\end{align}
We next estimate for all $\mu_0' \in B^2_r(\mu_0)$ the difference 
\begin{equation*}
\begin{aligned}
{E}
:=\big\| \zeta^\lambda-\zeta \big\|_{L^2} 
\leq E_1+E_2
~\mbox{where}~
& E_1:=\left \|  \left (D_{\xi}X_T^{\xi^{\lambda_3}} \right )^* \left [ \partial_x \delta_\mu \phi (\mu_T^{\lambda_1},X_T^{\lambda_2}  )- \partial_x \delta_\mu \phi (\mu_T,X_T^{\xi}  ) \right] \right \|_{L^2}
\\
& E_2:=  \left \| \left[ \left (D_{\xi}X_T^{\xi^{\lambda_3}} \right )^* - \left (D_{\xi}X_T^{\xi} \right )^*\right]  \partial_x \delta_\mu \phi (\mu_T,X_T^{\xi}  )\right \|_{L^2}.
\end{aligned}
\end{equation*}
Using  Theorem \ref{th:derivative_XtX0_wrt_X0} (iii),  Assumption \ref{ass:linear_fuunctional_derivative}(iii) (taking without loss of generality a concave modulus of continuity $\varpi_{\mu_T}$ therein), \eqref{eq:estimate_E_X_X0_wrt_initial_cond}, and Jensen's inequality, we have 
\begin{align*}
E_1&\leq C \left [ \left \|  \partial_x \delta_\mu \phi (\mu_T^{\lambda_1},X_T^{\lambda_2}  )- \partial_x \delta_\mu \phi (\mu_T,X_T^{\lambda_2} ) \right \|_{L^2}+\left \|  \partial_x \delta_\mu \phi (\mu_T,X_T^{\lambda_2}  )- \partial_x \delta_\mu \phi (\mu_T,X_T^{\xi}  ) \right \|_{L^2} \right]\\
& \leq C\left[ \varpi \left (\mathcal W_2(\mu_T^{\lambda_1},\mu_T)  \right) 
                    + \left \|  \varpi_{\mu_T} \left (|X_T^{\lambda_2} -X_T^{\xi} | \right )  \right \|_{L^2}\right] 
\\
& \leq C \left[\varpi \left (\left \| X_T^{\lambda_1} -X_T^{\xi}  \right \|_{L^2}  \right) 
                    + \varpi_{\mu_T} \left (\left \| X_T^{\lambda_2} -X_T^{\xi}   \right \|_{L^2}  \right)\right]\leq  \varpi \left ( \left \| \xi' -\xi \right \|_{L^2}  \right),
\end{align*}
where the moduli of continuity changed from line to line and we  dropped the dependence of $\varpi$ from $\mu_T$ since the letter is a fixed law (for fixed $\mu_0$).
By similar arguments, we may also estimate $E_2$ noting that, by Theorem \ref{th:derivative_XtX0_wrt_X0} (iii), we have $E_{2} \leq \varpi_{\xi} (\left \| \xi^{\lambda_3} -\xi \right \|_{L^2} )$, and we may then conclude that there exists a modulus of continuity $\hat\varpi$ such that 
\begin{equation}\label{eq:E}
E \leq \hat\varpi \left ( \left \| \xi' -\xi \right \|_{L^2}  \right)=\hat \varpi (\mathcal W_2(\mu_0',\mu_0) ) \leq \hat \varpi(r), \ \ \  \mu_0' \in B^2_r(\mu_0),
\end{equation}
where the last equality holds because $\P_{(\xi',\xi)}=:\pi^* \in \Pi(\mu_0,\mu'_0)$  is an  optimal coupling for $\mu'_0=\P_{\xi'}$ and $\mu_0=\P_{\xi}$. Plugging the last estimate in \eqref{eq:equality_incremental limit_varphi}, we obtain $\Delta_r \leq   \hat \varpi(r)+  \|\zeta \|_{L^2}$, which induces the required result by taking the  $\limsup_{r \to 0}$.
\medskip

\noindent \emph{Step 2.} We next show that $\|\zeta\|_{L^2} \leq \liminf_{r \searrow 0} \Delta_r$ to complete the proof of the theorem. Fix an arbitrary $\xi \in \mathcal{R}_{\mu_{0}}$, set $\hat \xi:=\xi+r \frac{\zeta}{ \|\zeta \|_{L^2} }$ and $ \hat\mu_0:=\P_{ \hat \xi}$. As $\zeta\in L^2(\mathcal F_0;\mathbb R^d),$ by Lemma \ref{lemma:growth_bounds_phi_der_phi} and Theorem \ref{th:derivative_XtX0_wrt_X0} (iii), we have 
 $ \hat \xi \in L^2(\mathcal F_0;\mathbb R^d)$. Moreover $\mathcal W_2( \hat \mu_0,  \mu_0) \leq \| \hat \xi-\xi\|_{L^2} = r,$ so that $ \hat \mu_0 \in B^2_r(\mu_0).$
Hence, by \eqref{eq:delta_Delta} and, using the same notations as in Step 1 with $\mu'_0:=\hat\mu_0$, we may write
\begin{align*}
\Delta_r
 \ge 
 \int_{[0,1]^3}  \frac{\E \big[ \big\langle  \zeta^\lambda, \hat \xi- \xi \big\rangle  \big]}{r\|\zeta\|_{L^2}}d\lambda
 =
\int_{[0,1]^3} \!\!\frac{\E \big[ \big\langle  \zeta^\lambda, \zeta \big\rangle \big]}{\|\zeta\|_{L^2}} d\lambda
& 
= 
\|\zeta\|_{L^2}
+ \int_{[0,1]^3}  \!\!\frac{\E \big[ \big\langle  \zeta^\lambda-\zeta, \zeta \big\rangle \big]}{\|\zeta\|_{L^2}}d\lambda
 \\
& \geq \|\zeta\|_{L^2}-\varpi (\| \hat\xi-\xi\|_{L^2} ) 
=  \|\zeta\|_{L^2} -\varpi (r),
\end{align*}
where the last inequality follows from \eqref{eq:E}. Taking the $\liminf_{r \to 0}$ completes the proof.
\end{proof}
}


 \section{Differentiability of solutions of McKean-Vlasov SDEs with respect to the initial datum}
 \label{sec:diff}
 
  \label{sec:derivative_process}
This section is dedicated to the proof of Theorem \ref{th:derivative_XtX0_wrt_X0} by essentially revisiting the arguments of \cite{buckdahn_li_peng_rainer}, pointing out some relevant features and completing them appropriately for our needs. 

\subsection{The auxiliary process $X_t^{x_0,\P_{\xi}}$} We recall that the idea 
of  \cite{buckdahn_li_peng_rainer} to study the differentiability of solutions of McKean-Vlasov SDEs with respect to the initial datum is to disentangle in the McKean-Vlasov SDE the dependence on the initial datum $\xi$ and the dependence of the coefficients on  the law $\P_{X_{s}}^{\xi}$. So,  one  starts by freezing the dependence of the coefficients on the flow of measures  $(\P_{X_s^{\xi}})_{s\in[0,T]}$  and  introduce the auxiliary  (standard) SDE in $X_{s}$
\begin{equation}\label{eq:sde_BLPR}
X_t=x_0+\int_0^t \hat b\left(s,  X_s\right) d Z_s, \quad  t \in[0, T],
\end{equation}
where $x_0\in\R^{d}$ and $\hat b\left(s, x\right):= b\left(x, \P_{X_{s}^{\xi}}\right) $.
Clearly, under our assumptions,   \eqref{eq:sde_BLPR} admits a unique solution  in the class $\mathcal{S}^2\left([0, T] ; \mathbb{R}^{d }\right)$, that we denote  by $X^{x_0,\xi}$. 
Notice that this solution only depends on the law of $\xi$: if $\P_{\xi}=\P_{\xi'}$, we have $\P_{X_t^{\xi}}=\P_{X_t^{\xi}}$ for every $t \in [0,T]$, so that  $X^{x_0,\xi}$ and $X^{x_0,\xi'}$ are indistinguishable. 
Then, we can define without ambiguity $X_t^{x_0,\P_{\xi}}:=X_t^{x_0,\xi}$. More generally, given $\hat \xi$ $\in L^{2}({\mathcal{F}_{0}})$,  we define without ambiguity
\begin{equation}\label{aas1}
X_t^{\hat\xi,{\xi}}(\omega):= X_t^{\hat\xi,\P_{\xi}}(\omega):= X_t^{\hat x_0,\P_\xi}(\omega)\mid_{\hat x_0=\hat \xi(\omega)} \quad \mbox{for a.e.} \ \omega \in \Omega, \quad  t \in [0,T].
\end{equation}
In particular
\begin{equation}\label{aas}
X_t^{\xi,{\xi}}(\omega):= X_t^{\xi,\P_{\xi}}(\omega)= X_t^{x_0,\P_\xi}(\omega)\mid_{x_0=\xi(\omega)} \quad \mbox{for a.e.} \ \omega \in \Omega, \quad  t \in [0,T].
\end{equation}
Clearly, $X_t^{\xi,\P_\xi}$ satisfies \eqref{eq:sde}; hence, 
by pathwise uniqueness of solutions to \eqref{eq:sde}, we have
\begin{equation}\label{aaa}
X^{\xi,\P_{\xi}}= X^{\xi}, \quad \P\textit{-a.s.}
\end{equation}
\begin{remark}\label{rem:conditional_expectation_wrt_eta_X0_FIRST}
Since the increments of the Brownian motion $B$ are independent of  $\mathcal{F}_{0}$,  we observe for later use that 
\begin{equation*}
\begin{aligned}
\E\left[    X_t^{\hat \xi, {\xi}} \,\big |\, \hat \xi=\hat x_{0}, \, \xi=x_0\right]=\E\left[  X_t^{\hat x_0, \delta_{x_0}}  \right],\quad  t\in [0,T], \ x_0, \hat x_{0} \in \mathbb R^d. 
\end{aligned}
\end{equation*}
In particular, taking $\hat \xi=\xi$, we obtain
\begin{equation*}
\begin{aligned}
\E\left[    X_t^{\xi, {\xi}} \,\big |\, \xi=x_0\right]= \E\left[    X_t^{\xi, \P_{\xi}} \,\big |\, \xi=x_0\right]=\E\left[  X_t^{x_0, \delta_{x_0}}  \right],\quad  t\in [0,T], \ x_0 \in \mathbb R^d. 
\end{aligned}
\end{equation*}
\end{remark}
\subsubsection{Derivative of $X_t^{x_0,\P_{\xi}}$ with respect to $x_0$ }
In this subsection we construct  the derivative of the auxiliary process $X_t^{x_0,\P_{\xi}}$ with respect to $x_0$. To this purpose, we introduce the tangent process  $\nabla X^{ x_0, \P_{\xi}}$ as the unique solution of
\begin{align}\label{eq:tangent_process}
&\nabla X_t^{ x_0, \P_{\xi}}= I_{d}+\int_0^t\partial_{x} b  \left(X_r^{x_0, \P_{\xi}}, \P_{X_r^{\xi}}\right)  \nabla  X_r^{x_0, \P_{\xi}} d Z_r ,
\end{align}
where $I_d$ denotes the identity matrix of $\R^{d\times d}$. Under Assumption  \ref{ass:derivatives_coeff}, for every $\xi \in L^2(\mathcal F_0; \mathbb R^{d })$, the map $x_0 \mapsto X^{x_0,\P_{\xi}} \in \mathcal{S}^2\left([0, T] ; \mathbb{R}^{d }\right)$ is continuously differentiable with respect to $x_0$ and
	\begin{equation}\label{limit}
\lim_{h \to 0} \frac 1{|h|}   \left \| X^{ x_0+h, \P_{\xi}}-X^{x_0, \P_{\xi}}-\nabla  X^{x_0, \P_{\xi}}  h\right \|_{ \mathcal S^2}=0.
\end{equation}
 Moreover,  there exists  $C >0$ such that, for every $t \in[0, T]$, every  $x_0, x_0^{\prime} \in \mathbb{R}^d$, and every  $\xi, \xi^{\prime} \in L^2\left(\mathcal F_0; \mathbb{R}^d\right)$, we have
	\begin{equation}\label{eq:tangent_lec_C}
\left \|\nabla  X^{x_0, \P_{\xi}}\right \|_{\mathcal S^2}^2  \leq C,
\end{equation}
and
	\begin{equation}\label{eq:tangent_continuity_est}
\left \|\nabla  X^{x_0, \P_{\xi}}-\nabla  X^{x_0', \P_{\xi'}} \right\|_{\mathcal S^2}^2 \leq C\left(\left|x_0-x_0^{\prime}\right|^2+\mathcal W_2(\P_{\xi},\P_{\xi'})^2\right),
\end{equation}
as proved in \cite[Lemma 4.1]{buckdahn_li_peng_rainer}. Hence,  the limit  in \eqref{limit} is uniform in $x_0,\xi$ and, in particular, there exists $C>0$ such that, 
	\begin{equation}\label{eq:differentiability_x0_estimate}
\left \| X^{ x_0+h, \P_{\xi}}-X^{x_0, \P_{\xi}}-\nabla  X^{x_0, \P_{\xi}}  h\right \|_{\mathcal S^2}\leq C |h|^2, \ \ \ \   \xi \in L^2(\mathcal F_0; \mathbb R^{d }),  \  x_0, h \in \mathbb R^d.
\end{equation}
As observed in \cite[Remark 4.1]{buckdahn_li_peng_rainer}, the process $\nabla  X^{\xi, \P_{\xi}}:=\left.\nabla  X^{x_0, \P_{\xi}}\right|_{x_0=\xi}$ is the unique solution in $\mathcal{S}^2\left([0, T] ; \mathbb{R}^{d \times d}\right)$ to the SDE 
\begin{equation}\label{eq:partial_x0_X_t_X0_X0}
\begin{aligned}
\nabla  X_t^{ \xi, \P_{\xi}}= & I_{d}+\int_0^t \partial_{x} b\left(X_r^{\xi}, \P_{X_r^{\xi}}\right) \nabla X_r^{\xi, \P_{\xi}} d Z_r.
\end{aligned}
\end{equation}

\subsubsection{Derivative of $X_t^{ x_0, \P_{\xi}}$ with respect to $\xi$ }\label{subsec:YtX0_eta}

Let us  consider a copy of  $(\Omega, \mathcal{F}, \mathbb{F},  \P)$, which we denote by $(  \tilde{\Omega}, \tilde{\mathcal{F}}, \tilde{\mathbb{F}},  \tilde{\P})$. In this setting, given any random variable $X$ defined on $(  \Omega, \mathcal{F}, \mathbb{F},  \P)$, we may consider a random variable $\tilde X$, which is identical to $X$, but it is defined on $(  \tilde{\Omega}, \tilde{\mathcal{F}}, \tilde{\mathbb{F}},  \tilde{\P})$; hence, the expectation $\tilde\E[\tilde X]:=\int_{\tilde \Omega} \tilde X d\tilde P$ will apply only to $\tilde X$.

For $\xi, \eta \in L^2(\mathcal F_0;\mathbb R^d)$, consider the family of  mean-field SDEs with values in $\mathbb R^d$
\begin{small}
\begin{equation}\label{eq:Y_t_X0_eta}
\begin{aligned} & Y_t^{\xi}(\eta)= \int_0^t \left( \partial_{x}b \left(X_r^{ \xi }, \P_{X_r^{\xi}}\right) Y_r^{\xi}(\eta)  +  \tilde{\mathbb {E}}\left[\partial_{\tilde x}  \delta_{\mu} b \left(X_r^{\xi, \P_{\xi}}, \P_{X_r^{\xi}}, \tilde{X}_r^{ \tilde \xi}\right) \left(\nabla  \tilde{X}_r^{ \tilde \xi,\P_{\xi}} \tilde{\eta}+\tilde{Y}_r^{ \tilde \xi}(\tilde{\eta})\right)\right] \right) d Z_r,
\end{aligned}
\end{equation}
\end{small}
where $( \tilde \xi, \tilde \eta, \tilde {B},\tilde X^{\tilde \xi},\tilde{Y}^{ \tilde \xi}(\tilde{\eta}))$ is a copy of $(\xi, \eta, B,X^{\xi},Y^{ \xi}(\eta))$ defined on the probability space $(  \tilde{\Omega}, \tilde{\mathcal{F}}, \tilde{\mathbb{F}},  \tilde{\P})$. In addition, given a copy $\tilde{Y}_r^{ \tilde \xi}(\tilde{\eta})$ of a solution  $Y_r^{ \xi}(\eta)$ to \eqref{eq:Y_t_X0_eta}, we also consider the standard SDE for all $x_0 \in \mathbb R^d$:
\begin{small}
\begin{equation*}
\begin{aligned} 
&Y_t^{x_0, \xi}(\eta)= \int_0^t\left ( \partial_{x} b \left(X_r^{ x_0,\xi }, \P_{X_r^{\xi}}\right) Y_r^{x_0, \xi}(\eta)   +  \tilde{\mathbb {E}}\left[\partial_{\tilde x} \delta_{\mu} b \left(X_r^{x_0, \P_{\xi}}, \P_{X_r^{\xi}}, \tilde{X}_r^{ \tilde \xi}\right) \left(\nabla \tilde{X}_r^{ \tilde \xi,\P_{\xi}} \tilde{\eta}+\tilde{Y}_r^{ \tilde \xi}(\tilde{\eta})\right)\right] \right) d Z_r ,
\end{aligned}
\end{equation*}
\end{small}
As in \cite[Eqs.\,(4.15)--(4.16)]{buckdahn_li_peng_rainer}, due to the boundedness of the derivatives of $b,\sigma$ (see Assumption \ref{ass:derivatives_coeff}), the system of  equations above admits a unique solution $\left(Y^{\xi}(\eta), Y^{x_0,\xi}(\eta) \right) \in \mathcal{S}^2\left([0, T] ; \mathbb{R}^{ d} \times \mathbb{R}^{d} \right)$. By uniqueness of solutions to \eqref{eq:Y_t_X0_eta}, we have
\begin{equation}\label{eq:Y_X0,X0=Y_X0}
Y_t^{\xi, \xi}(\eta):=Y_t^{x_0, \xi}(\eta)\mid_{x_0=\xi} = Y_t^{ \xi}(\eta).
\end{equation}
Notice that $Y_t^{\xi}(\eta), Y_t^{x_0,\xi}(\eta)$ are linear in $\eta$ (i.e. $\alpha_1 Y_t^{\xi}(\eta)+\alpha_2Y_t^{\xi}(\eta')$ satisfies \eqref{eq:Y_t_X0_eta} for $\alpha_1\eta+\alpha_2 \eta'$ and then, thanks to uniqueness of solutions we have $\alpha_1Y_t^{\xi}(\eta)+\alpha_2Y_t^{\xi}(\eta')=Y_t^{x_0,\xi}(\alpha_1\eta+\alpha_2\eta')$); hence we will write $Y^{\xi}_t \eta$ and  $Y^{x_0,\xi}_t\eta$. Moreover, as in \cite[Eq. (4.17)]{buckdahn_li_peng_rainer}, there exists $C>0$ such that for each $x_0 \in \mathbb R^d$ and $ \xi, \eta \in L^2(\Omega;\mathbb R^d)$, one has
$$
\|Y_t^{x_0, \xi}\eta\|_{L^2} \leq \left \|Y^{x_0, \xi}\eta \right\|_{\mathcal S^2} \leq C \|\eta\|_{L^2},
$$
for every $ t \in [0,T]$; this implies
$
Y_t^{x_0, \xi}\in \mathcal L\left(L^2(\mathcal F_0;\mathbb R^d),L^2 (\mathcal F_t;\mathbb R^d)\right)$, for every $t \in[0, T] .$

Now, let 
 $x_0,y \in \mathbb R^d, \xi \in L^2(\mathcal F_0;\mathbb R^d)$ and consider also the family of mean-field SDEs with values in $\mathbb R^{d \times d}$
\begin{small}
$$
\begin{aligned}
& U_t^{\xi}(y)= \int_0^t \left(\partial_{x} b \left(X_r^{\xi}, \P_{X_r^{\xi}}\right) U_r^{\xi}(y)  + \tilde{\E}\left[\partial_{\tilde x}  \delta_{\mu}  b \left(X_r^{\xi}, \P_{X_r^{\xi}}, \tilde{X}_r^{y, \P_{\xi}}\right) \left(\nabla  \tilde{X}_r^{y, \P_{\xi}}+ \tilde{U}_r^{ \tilde \xi}(y) \right )\right] \right) d Z_r,
\end{aligned}
$$
\end{small}
where $\left(\tilde \xi, \tilde{B}, \tilde X^{x_0,\P_{\xi}}, \tilde{U}^{ \tilde \xi}(y)\right)$ is a  copy of  $\left(\xi,B,X^{x_0,\P_{\xi}},U^{ \xi}(y), \right)$ defined on the probability space $(  \tilde{\Omega}, \tilde{\mathcal{F}}, (\tilde{\mathcal{F}}_t)_{t \in [0,T]},  \tilde{\P})$.
Given a copy  $ \tilde U_t^{\tilde \xi}(y)$ of a solution $ U_t^{\xi}(y)$  to such equation, consider also the standard SDE
\begin{small}
$$
\begin{aligned}
& U_t^{x_0, \xi}(y)=\int_0^t \left(\partial_{x} b \left(X_r^{ x_0, \P_{\xi}}, \P_{X_r^{\xi}}\right) U_r^{x_0, \xi}(y)   +  \tilde{\E}\left[  \partial_{\tilde x}  \delta_{\mu}  b\left(X_r^{ x_0, \P_{\xi}}, \P_{X_r^{\xi}}, \tilde{X}_r^{y, \P_{\xi}}\right)\left( \nabla  \tilde{X}_r^{y, \P_{\xi}} + \tilde{U}_r^{ \tilde \xi}(y)\right) \right] \right)d Z_r,
\end{aligned}
$$
\end{small}
\begin{flushleft}
As in \cite[Eqs.\,(4.48)--(4.49)]{buckdahn_li_peng_rainer}, the system of equations above admits a unique solution $\left(U^{\xi}_{\cdot}(\eta), U^{x_0,\xi}_{\cdot}(y) \right) \in \mathcal{S}^2\left([0, T] ; \mathbb{R}^{d \times d } \times \mathbb{R}^{d  \times d} \right)$ and
we have again
\end{flushleft}
\begin{equation}\label{eq:U_X0,X0=U_X0}
U_t^{ \xi,\xi}(y):=\left.U_t^{ x_0,\xi}(y)\right|_{x_0=\xi}=U_t^{ \xi}(y).
\end{equation}
By \cite[Lemma 4.3]{buckdahn_li_peng_rainer}, under Assumption \ref{ass:derivatives_coeff},
 there exists $C>0$ such that for every $t\in [0,T],x_0,x_0' \in \mathbb R^d,$ $\xi,\xi' \in L^2(\mathcal F_0;\mathbb R^d)$,
\begin{equation}\label{eq:est_U_t^x_0,X0}
\left \|U^{x_0, \xi}(y)\right\|_{\mathcal S^2} \leq C, \quad \left \| U^{x_0, \xi}(y)-U^{ x_0^{\prime}, \xi^{\prime}}\left(y^{\prime}\right)\right \|_{\mathcal S^2} \leq C\left(\left|x_0-x_0^{\prime}\right|+\left|y-y^{\prime}\right|+\mathcal W_2\left(\P_{\xi}, \P_{\xi^{\prime}}\right)\right) .
\end{equation}
Due to the result above,  we can define without ambiguity 
$$U_t^{x_0, \P_{\xi}}(y):=U_t^{x_0, \xi}(y).$$
We now have all ingredients to state  the Fr\'echet differentiability of $ X_t^{ x_0, \P_{\xi}} $ with respect to $\xi.$
\begin{lemma}\label{lemma:derivative_X_t_x0_X0_wrt_X0}
Let Assumption  \ref{ass:derivatives_coeff} hold. 
\\
{\rm (i)} For all $(t,x_0) \!\in\! [0,T]\times\mathbb R^d$, the map $\xi\!\in\! L^2(\mathcal F_0; \mathbb R^{d }) \longmapsto X_t^{ x_0, \P_{\xi}}\!\in\! L^2(\mathcal F_t; \mathbb R^{d })$
is continuously Fr\'echet differentiable; i.e.   there exists a continuous map
$ \xi\mapsto \partial_{\xi} X_t^{x_0, \P_{\xi}} \in \mathcal L \left( L^2(\mathcal F_0;\mathbb R^{d}); L^2(\mathcal F_t;\mathbb R^{d })\right),$ s.t. 
$$
\lim_{\|\eta\|_{L^2}  \rightarrow 0 } \frac 1{\|\eta\|_{L^2} }   \left  \|  X_t^{ x_0, \P_{\xi+\eta}}-X_t^{x_0, \P_{\xi}}-\partial_{\xi} X_{t}^{x_0, \P_{\xi}}  \eta \right \| _{L^2} =0, \ \ \ \  \xi \in L^2(\mathcal F_0;\mathbb R^d).
$$
{\rm (ii)} The Fréchet derivative $\partial_{\xi} X_{t}^{x_0, \P_{\xi}}  $ satisfies
\begin{equation}\label{eq:partial_X0_X_t_x_0,P_X0=Y=EU}
\partial_{\xi} X_{t}^{x_0, \P_{\xi}}  \eta =Y_t^{x_0, \xi}\eta=\tilde{\E}\left[U_t^{x_0, \P_{\xi}}( \tilde{\xi})  \tilde{\eta}\right] \quad  t \in[0,T], \  \eta \in L^2(\mathcal F_0;\mathbb R^d).
\end{equation}
{\rm (iii)} There exists $C>0$ such that
\begin{align*}
&\left  \|\partial_{\xi} X_t^{x_0, \P_{\xi}} \right \|_{\mathcal L} \leq C,\ \ \ \  \ \ \  t\in [0,T], \  x_0 \in \mathbb R^d, \   \xi \in L^2(\mathcal F_0;\mathbb R^d).
\end{align*}
{\rm (iv)}  There exists $C>0$ such that for every $t\in [0,T], \  x_0,x_0' \in \mathbb R^d,  \  \xi,\xi' \in L^2(\mathcal F_0;\mathbb R^d)$,
\begin{align*}
\left  \|\partial_{\xi} X_t^{x_0, \P_{\xi}}-\partial_{\xi} X_t^{x_0', \P_{\xi'}}   \right \|_{\mathcal L} & \leq C \left( |x_0-x_0'|+\mathcal W_2\left(\P_{\xi},\P_{\xi'}\right)\right).
\end{align*}
In particular, it follows that the limit in (i) is uniform in $t, x_0,\xi$ and there exists $C>0$ such that, for every  $x_0 \in \mathbb R^d$, $t \in [0,T]$, and $\xi \in L^2(\mathcal F_0;\mathbb R^d)$,
$$
 \left  \|  X_t^{ x_0, \P_{\xi+\eta}}-X_t^{x_0, \P_{\xi}}-\partial_{\xi} X_{t}^{x_0, \P_{\xi}}  \eta \right \| _{L^2} \leq C \|\eta\|_{L^2}^2 .
$$
\end{lemma}
\begin{proof}
 (i), (ii), and (iii) follows directly from \cite[Proposition 4.2]{buckdahn_li_peng_rainer}. As for   (iv), 
 notice that the estimate in  (iv) with $x_0=x_0'$ has been shown in the proof  of \cite[Proposition 4.2]{buckdahn_li_peng_rainer}, using \ref{eq:est_U_t^x_0,X0}. Such calculation can be extended to prove (iv) for general $x_0 \neq  x_0'$.  
\end{proof}
With usual notations, we have
\begin{small}
\begin{equation}\label{eq:partial_X0_X_X:0=Y_X0}
\begin{aligned}
\partial_{\xi} X_t^{\xi, \P_{\xi}}\eta &:=\partial_{\xi} X_t^{x_0, \P_{\xi}}  \Big |_{x_0=\xi} \eta
=\tilde{\E}\left[U_t^{ x_0,\P_{\xi}}( \tilde \xi)\tilde \eta  \right]  \Big |_{x_0=\xi} =\tilde{\E}\left[U_t^{ \xi}( \tilde \xi) \tilde \eta \right].
 \end{aligned}
\end{equation}
\end{small}

\begin{remark}\label{rem:conditional_expectation_wrt_eta_X0}
As for Remark \ref{rem:conditional_expectation_wrt_eta_X0_FIRST}, 
 for every $x_0 ,h\in \mathbb R^d$, we have
\begin{equation}\label{eq_conditional_expectation_wrt_eta_X0}
\begin{aligned}
\E\left[   \partial_{\xi} X_t^{\xi, \P_{\xi}} \eta \,\big |( \xi,\eta)=(x_0,h)\right]=\E\left[ \partial_{\xi} X_t^{x_0, \delta_{x_0}}  h\right]=\E\left[ Y_t^{x_0} h\right]=\E\left[ \tilde E[U_t^{x_0}(x_0)h] \right].
\end{aligned}
\end{equation}
\end{remark}

In view of  the equality $X_t^{\xi}=X_t^{\xi, \P_{\xi}}$, we warn the reader not to confuse $\partial_{\xi} X_t^{\xi, \P_{\xi}}$, i.e. the (partial) derivative with respect to $\xi$  of $X_t^{x_0,\P_{\xi}}$ evaluated at $x_0=\xi$, with the (total) derivative of $X_t^{\xi}=X_t^{\xi, \P_{\xi}}$ with respect to $\xi$, which we will denote by $D_{\xi}X_t^{\xi}$.

\subsection{Derivative of $X_t^{\xi}$ with respect to $\xi$ }\label{sec:derivativeBuck}
Recall \eqref{aas} and consider the map
\begin{equation}\label{pooi}
\xi\in L^2(\mathcal F_0;\mathbb R^{d})\longmapsto X_t^{\xi}=X_t^{\xi,\P_{\xi}}\in L^2(\mathcal F_t;\mathbb R^{d }).
\end{equation}
 Formally, computing the derivative of this map,  
it is natural to expect (see also \cite[Eq.\,(4.8),\,p. 840]{buckdahn_li_peng_rainer}) that it is the operator $D_{\xi} X_t^{\xi} \in \mathcal L\left (L^2(\mathcal F_0;\mathbb R^d); L^2(\mathcal F_t;\mathbb R^d)\right)$ given by
\begin{equation}\label{GD}
\left([D_{\xi}X_t^{\xi}]\eta \right) (\omega):=\nabla  X_t^{\xi, \P_{\xi}}(\omega)\eta(\omega)  +\left[ \partial_{\xi} X_{t}^{\xi, \P_{\xi}}  \eta \right] (\omega)  , \ \ \  \eta \in L^2(\mathcal F_0;\mathbb R^d).
\end{equation}
 In this section, we will show that the operator $D_{\xi}X_t^{\xi}$ defined above is indeed the Gateaux derivative of \eqref{pooi}.
Let $\eta \in L^2(\mathcal F_0;\mathbb R^d)$. By  \eqref{eq:partial_x0_X_t_X0_X0} and  \eqref{eq:Y_t_X0_eta},  we have,  for every $\xi \in L^2(\mathcal F_0;\mathbb R^d)$,
\begin{align}\label{deff}
D_{\xi}X_t^{\xi}\eta=\ & \eta+\int_0^t  \left( \partial_{x}b  (X_s^{\xi},\P_{X_s^{\xi}}) D_{\xi}X_s^{\xi}\eta + \tilde{ \E}\left[ \partial_{\tilde x} \delta_{\mu} b (X_s^{\xi},\P_{X_s^{\xi}}, \tilde  X_s^{\tilde \xi}) D_{\xi}\tilde X_s^{\tilde \xi} \tilde \eta \right]\right) dZ_s,
\end{align}
where $\tilde X_s^{\tilde  \xi}$ is the solution to \eqref{eq:sde}  on $(\tilde \Omega, \tilde{\mathcal{F}}, \tilde {\mathcal{F}}_t, \tilde  B_t, \tilde {\P})_{t \in [0,T]}$  with initial condition $\tilde \xi$ such that $\tilde  {\P}_{\tilde \xi}=\mu_0$. Moreover,  due to  boundedness of the derivatives of $b$ (cf. Assumption \ref{ass:derivatives_coeff}), we have that $D_{\xi}X_t^{\xi}\eta$ is the unique solution to the equation above. \\\\


\begin{proof}[Proof of Theorem \ref{th:derivative_XtX0_wrt_X0}]

\emph{Step 1}. We start by showing  the estimates of Claim (i) for the operator $D_{\xi}X_t^{\xi}$  defined in  \eqref{deff}. Indeed, we have
\begin{align*}
\left \| D_{\xi} X_{t}^{\xi} \eta \right \|_{L^2} \leq  \left \|\nabla  X_t^{\xi, \P_{\xi}} \eta \right  \|_{L^2}  + \left \|\partial_{\xi} X_{t}^{\xi, \P_{\xi}}  \eta\right \|_{L^2}, \quad  \eta \in L^2(\mathcal F_0;\mathbb R^d).
\end{align*}
For the first term on the right-hand-side we have
\begin{align*}
\left \|\nabla  X_t^{\xi, \P_{\xi}} \eta \right  \|_{L^2}= \left( \E \left [\E \left [ \left |  \nabla  X_t^{\xi, \P_{\xi}} \eta  \right |^2 \bigg | \left( \xi ,\eta \right) \right]  \right] \right)^{1/2}  .
\end{align*}
By the independence of $(\xi,\eta )$ (which are $\mathcal F_0-$measurable) and the increments of $B_t$ (which are independent of $\mathcal F_0$),  and using \eqref{eq:tangent_lec_C},  for all $ x_0,h \in \mathbb R^d$,
\begin{align*}
\E \left [ \left |  \nabla  X_t^{\xi, \P_{\xi}} \eta  \right |^2 \bigg | \left( \xi ,\eta \right)=(x_0,h) \right]  =\E \left [ \left |  \nabla  X_t^{x_0, \delta_{x_0}} h \right |^2 \right] \leq C |h|^2,
\end{align*}
so that $\left \| \nabla  X_t^{\xi, \P_{\xi}} \eta \right \|_{L^2} \leq C \left \| \eta \right \|_{L^2} $.  For the second term on the right-hand-side, we have
\begin{align*}
\left \|\partial_{\xi} X_t^{\xi, \P_{\xi}} \eta \right  \|_{L^2}= \left( \E \left [\E \left [ \left |  \partial_{\xi} X_t^{\xi, \P_{\xi}} \eta  \right |^2 \bigg | \left( \xi ,\eta \right) \right]  \right] \right)^{1/2}  
\end{align*}
By Remark \ref{rem:conditional_expectation_wrt_eta_X0} and  Lemma \ref{lemma:derivative_X_t_x0_X0_wrt_X0}, for all
$x_0,h \in \mathbb R^d$
\begin{align*}
\E \left [ \left |  \partial_{\xi} X_t^{\xi, \P_{\xi}} \eta  \right |^2 \bigg | \left( \xi ,\eta \right)=(x_0,h) \right]  =\E \left [ \left |  \partial_{\xi} X_t^{x_0, \delta_{x_0}} h \right |^2 \right] \leq C |h|^2,
\end{align*}
so that $\left \| \partial_{\xi} X_{t}^{\xi,\P_{\xi}} \eta \right \|_{L^2} \leq C \left \| \eta \right \|_{L^2} $. Hence 
$$\left \| D_{\xi} X_{t}^{\xi} \eta \right \|_{L^2} \leq C \left \| \eta \right \|_{L^2} ,$$ 
so that we have the claim.

Next, we prove the uniform strong continuity estimate. For $ \eta, \xi,\xi'\in L^2(\mathcal F_0;\mathbb R^d)$, we have:
\begin{equation}\label{eq:est_DX0_XtX0-DX0_XtX0prime}
\begin{aligned}
\left \| \left (D_{\xi} X_{t}^{\xi} - D_{\xi} X_{t}^{\xi'} \right) \eta \right \|_{L^2}  
\leq 
E_t^1(\eta)+E_t^2 (\eta),
~\mbox{where}~
& E_t^1(\eta):=
\left \| \left (\nabla X_t^{\xi, \P_{\xi}} - \nabla X_t^{\xi', \P_{\xi'}} \right) \eta  \right \|_{L^2}  
\\
& E_t^2 (\eta):=
\left \| \left ( \partial_{\xi} X_t^{\xi, \P_{\xi}}  - \partial_{\xi} X_t^{\xi', \P_{\xi'}} \right) \eta\right \|_{L^2}.
\end{aligned}
\end{equation}
\begin{itemize}
\item Consider $E_t^1(\eta) $. By the independence of $(\xi,\xi')$ and the increments of $B_t$, it follows from \eqref{eq:tangent_continuity_est} that for all  $ x_0,x_0' \in \mathbb R^d$
\begin{align*}
\E \left [ \left |  \nabla  X_t^{\xi, \P_{\xi}} - \nabla  X_t^{\xi', \P_{\xi'}}  \right |^2 \bigg | \left( \xi,\xi'   \right)=(x_0,x_0') \right] 
&=\E \left [ \left |\nabla  X_t^{x_0, \delta_{x_0}} - \nabla  X_t^{x_0', \delta_{x_0'}}  \right |^2 \right] \\
& \leq C \left(\left|x_0-x_0^{\prime}\right|^2+\mathcal W_2\left(\delta_{x_0},\delta_{x_0'}\right)^2\right)\leq C \left|x_0-x_0^{\prime}\right|^2,
\end{align*}
which implies that 
\begin{equation}\label{eq:partial_x0XtX0-partial_x0XtX0'}
\left \| \nabla  X_t^{\xi, \P_{\xi}} - \nabla  X_t^{\xi', \P_{\xi'}}  \right \|_{L^2} \leq  C\| \xi-\xi' \|_{L^2},
\end{equation}
and therefore for $\eta \in L^\infty(\mathcal F_0;\mathbb R^d)$:
\begin{align}\label{eq:estm_Et1_eta_in_L_infty}
E_t^1 (\eta)\leq \left \| \nabla  X_t^{\xi, \P_{\xi}} - \nabla  X_t^{\xi', \P_{\xi'}}  \right \|_{L^2}  \|\eta\|_{L^\infty }  \leq  C\| \xi-\xi' \|_{L^2} \|\eta\|_{L^\infty } .
\end{align}

Let now $\eta \in L^2(\mathcal F_0;\mathbb R^d).$ By considering $\eta_N := \eta I_{|\eta| \leq N} \in L^{\infty}(\mathcal F_0;\mathbb R^d)$ (so that   $ \|\eta_N - \eta\|_{L^2} \to 0),$ we have
\begin{align*}
E_t^1 (\eta) \leq E_t^1 (\eta-\eta_N) +  E_t^1 (\eta_N).
\end{align*}
Using \eqref{eq:tangent_lec_C}, we have $E_t^1 (\eta-\eta_N) \leq C \|\eta_N - \eta\|_{L^2}.$ Moreover, by  \eqref{eq:estm_Et1_eta_in_L_infty}, we have $E_t^1 (\eta_N ) \leq C \| \xi-\xi' \|_{L^2}  \|\eta_N\|_{L^\infty}.$ Hence
\begin{align*}
E_t^1 (\eta) \leq C \|\eta_N - \eta\|_{L^2} +C \| \xi-\xi' \|_{L^2}  \|\eta_N\|_{L^\infty}.
\end{align*}
Let $\epsilon>0$. Take first $\bar N$ such that $C \|\eta_{\bar N} - \eta\|_{L^2} < \epsilon/2$. Then  there exists $\delta>0$ such that $C\| \xi-\xi' \|_{L^2}  \|\eta_{\bar N}\|_{L^\infty}< \epsilon/2$, for all $ \| \xi-\xi' \|_{L^2}  < \delta.$ 
Hence,  for all $ \eta \in L^2(\mathcal F_0;\mathbb R^d)$ there exists a modulus of continuity $\varpi_\eta$ such that $E_t^1 (\eta) \leq \varpi_\eta( \| \xi-\xi' \|_{L^2} ) $. 
\item Consider $E_t^2(\eta)$. Let $\eta \in L^2 (\mathcal F_0;\mathbb R^d)$. 
By \eqref{eq:partial_X0_X_X:0=Y_X0}, we have
\begin{align*}
E_t^2(\eta)  & = \left \| \tilde{\mathbb E}\left[ \left  (U_t^{ \xi,\P_{\xi}}(\tilde{\xi}) -U_t^{ \xi',\P_{\xi'}}(\tilde{\xi'}) \right ) \tilde{\eta}\right] \right \|_{L^2}  \leq \left \| \left ( \tilde{\mathbb E}\left[ \left  | U_t^{ \xi,\P_{\xi}}(\tilde{\xi}) -U_t^{ \xi',\P_{\xi'}}(\tilde{\xi'}) \right |^2 \right] \right)^{1/2} \right \|_{L^2}   \| \tilde \eta \|_{L^2}\\
&= \left ( \mathbb E \left [ \mathbb E \left [ \tilde{\mathbb E}\left[ \left  | U_t^{ \xi,\P_{\xi}}(\tilde{\xi}) -U_t^{\xi', \P_{\xi'}}(\tilde{\xi'}) \right |^2  \right]   \bigg | \left( \xi,\xi'  \right) \right]    \right ] \right )^{1/2}  \| \eta \|_{L^2} .
\end{align*}
As in  Remark \ref{rem:conditional_expectation_wrt_eta_X0} and using Lemma \ref{lemma:derivative_X_t_x0_X0_wrt_X0},  for all $ x_0,x_0' \in \mathbb R^d$,
\begin{align*}
 \mathbb E \left [  \tilde{\mathbb E}\left[ \left  | U_t^{ \xi,\P_{\xi}}\!(\tilde{\xi}) -U_t^{\xi', \P_{\xi'}}\!(\tilde{\xi'}) \right |^2  \right]   \bigg | \left( \xi,\xi'  \right)\!=\!(x_0,x_0')  \right] 
 =\mathbb E \left [ \tilde{\mathbb E}\left[ \left  | U_t^{ x_0,\delta_{x_0}}(x_0) -U_t^{x_0', \delta_{x_0'}}(x_0') \right |^2  \right]    \right]  \leq C|x_0-x_0'|^2.
\end{align*}
It follows
$E_t^2(\eta) \leq C  \| \xi-\xi' \|_{L^2} \| \eta \|_{L^2}$. Notice that we have actually proved that  
\begin{equation}\label{eq:norm_estimate_DX0_XtX0-DX0_XtX0prime}
\left \|  \partial_{\xi} X_t^{\xi, \P_{\xi}}  - \partial_{\xi} X_t^{\xi', \P_{\xi'}} \right \|_{\mathcal L}\leq C  \| \xi-\xi' \|_{L^2},
\end{equation}
\end{itemize} 
a stronger estimate implying the required claim.\\
\emph{Step 2.}
 We now complete the proof of Claim (i), i.e. that we have, for $D_{\xi} X_{t}^{\xi}$ defined  in \eqref{deff},
\begin{align}\label{eq:limit_directional_derivative}
& I_t (|r|,\eta) :=  \frac 1 {|r|}   \left  \|  X_t^{ \xi+r \eta}-X_t^{\xi}-r D_{\xi} X_{t}^{\xi} \eta \right\|_{L^2} \xrightarrow{r \to 0}0, \quad  \eta \in L^2(\mathcal F_0;\mathbb R^d).
\end{align}

We first provide an estimate  when $\eta \in L^4(\mathcal F_0;\mathbb R^d) $ and  then  use  a density argument.
Given  $r \in \mathbb \R$ and $\eta \in L^4(\mathcal F_0;\mathbb R^d)$, we have
\begin{align*}
& I_t (|r|,\eta) = \frac 1 {|r|}   \left  \|  X_t^{\xi+r\eta, \P_{\xi+r\eta}}-X_t^{\xi,\P_{\xi}}-\left[  r\nabla X_t^{\xi, \P_{\xi}}  \eta + r \partial_{\xi} X_{t}^{\xi, \P_{\xi}} \eta \right]  \right\|_{L^2}\\
&\leq \frac 1 {|r|}    \left  \|  X_t^{\xi+r \eta, \P_{\xi+r\eta}} - X_t^{\xi, \P_{\xi+ r \eta}}-r  \nabla  X_t^{\xi, \P_{\xi}} \eta  \right\|_{L^2} \\
&\quad + \frac 1 {|r|}  \left  \| X_t^{\xi, \P_{\xi+r \eta}}-X_t^{\xi,\P_{\xi}}-r \partial_{\xi} X_{t}^{\xi, \P_{\xi}}  \eta\right\|_{L^2}=: I^1_t (|r|,\eta) +I^2_t (|r|,\eta) .
\end{align*}
\begin{itemize}
\item[--] As for $I_t^1(|r|,\eta)$, we have
\begin{align*}
I_t^1 (|r|,\eta) = \frac 1 {|r|}  \left( \E \left [\E \left [ \left |  X_t^{\xi+r \eta, \P_{\xi+r \eta}} - X_t^{\xi, \P_{\xi+ r\eta}}- r \nabla  X_t^{\xi, \P_{\xi+r \eta}} \eta \right |^2 \bigg | \left( \xi ,\eta \right) \right]  \right] \right)^{1/2} .
\end{align*}
Since  $(\xi,\eta )$  is independent  of  $B_t$ and using Remarks \ref{rem:conditional_expectation_wrt_eta_X0_FIRST}, \ref{rem:conditional_expectation_wrt_eta_X0}, as wells as \eqref{limit},  \eqref{eq:tangent_continuity_est}, we have for every  $ x_0,h \in \mathbb R^d$
\begin{align*}
 & \E  \left [ \left |  X_t^{\xi+r \eta, \P_{\xi+r \eta}} - X_t^{\xi, \P_{\xi+r \eta}}- r \nabla  X_t^{\xi, \P_{\xi}} \eta \right |^2 \bigg | \left( \xi ,\eta \right)=(x_0,h)\right]\\
&=\E \left [ \left |  X_t^{x_0+r h, \P_{x_0+r h }} - X_t^{x_0, \P_{x_0+r h }}- r \nabla  X_t^{x_0, \delta_{x_0}} h\right |^2 \right]\\
&=\E \left [ \left | \int_0^1  \nabla  X_t^{x_0+\theta r h , \P_{x_0+r h}} rh \  d\theta - r \nabla  X_t^{x_0, \delta_{x_0}} h\right |^2 \right]\\
& \leq C \left (|r|^2 |h|^2+\mathcal W_2( \P_{x_0+rh},  \delta_{x_0})^2 \right ) |r|^2 |h|^2  \leq C |r|^4 |h|^4 ,
\end{align*}
for $C>0$ independent of  $x_0, |r|, |h|.$
 Hence $I^1_t (|r|,\eta)  \leq C |r| \left(\E \left[|\eta|^4 \right]\right)^{1/2} = C  |r| \|\eta\|_{L^4}^2$.
 \medskip
\item[--] As for $I^2_t(|r|,\eta)$, we have
$$I^2_t (|r|,\eta)  = \frac 1 {|r|}   \left ( \E \left [  \E \Big[ \left | X_t^{\xi, \P_{\xi+r \eta}}-X_t^{\xi,\P_{\xi}}-r\partial_{\xi} X_{t}^{\xi, \P_{\xi}} \eta \right |^2 \, \big | \, (\xi,\eta) \Big] \right] \right)^{1/2}. $$ 
As before,  using now  Remarks \ref{rem:conditional_expectation_wrt_eta_X0_FIRST}, \ref{rem:conditional_expectation_wrt_eta_X0}, and Lemma \ref{lemma:derivative_X_t_x0_X0_wrt_X0}, we have   for every $x_0,h \in \mathbb R^d$ 
\begin{align*} 
& \E  \left [ \left | X_t^{\xi, \P_{\xi+r \eta}}-X_t^{\xi,\P_{\xi}}-r\partial_{\xi} X_{t}^{\xi, \P_{\xi}} \eta \right |^2 \bigg |  (\xi,\eta)=(x_0 , h) \right] \\
&=\E \left [ \left | X_t^{x_0, \P_{x_0+r h }}-X_t^{x_0,\delta_{x_0}}-r \partial_{\xi} X_{t}^{x_0, \delta_{x_0}} h\right |^2  \right] \leq C |r|^4 |h|^4 , 
 \end{align*}
 where $C$ is independent of $x_0, |r|, |h|.$  Hence, $I^2_t  (|r|,\eta)  \leq C |r| \left(\E \left[|\eta|^4 \right]\right)^{1/2} = C  |r| \|\eta\|_{L^4}^2.$
 \end{itemize}
We therefore conclude that there exists $C>0$ (independent of $\eta$)  such that
\begin{equation}\label{eq:I_leq_C_r_eta_L4}
I_t  (|r|,\eta)  \leq  C  |r| \|\eta\|_{L^4}^2.
\end{equation}

Now, let  $\eta \in L^2(\mathcal F_0;\mathbb R^d) $ and  define $\eta_N := \eta I_{|\eta| \leq N} \in L^{\infty}(\mathcal F_0;\mathbb R^d)$. Clearly,  we have $ \|\eta_N - \eta\|_{L^2} \to 0$ as $N\to\infty$. By \eqref{eq:estimate_E_X_X0_wrt_initial_cond}, the estimate on the operator norm of $D_{\xi}X_t^{\xi}$, and \eqref{eq:I_leq_C_r_eta_L4}, we have
\begin{small}
\begin{align*} 
I_t  (|r|,\eta)  & \leq  \frac 1 {|r|}   \Big  \|  X_t^{\xi+r\eta }-X_t^{\xi}-r D_{\xi} X_{t}^{\xi} \eta - \left \{  X_t^{\xi+r\eta_N}-X_t^{\xi} -r D_{\xi} X_{t}^{\xi} \eta_N  \right \}  \Big \|_{L^2}+I_t  (|r|,\eta_N) \\
& \leq \frac 1 {|r|}   \Big  \|  X_t^{\xi+r\eta} -  X_t^{\xi+r\eta_N}  \Big  \|_{L^2}  +   \Big  \|  D_{\xi} X_{t}^{\xi} ( \eta - \eta_N)   \Big  \|_{L^2}+I_t  (|r|,\eta_N)\\
& \leq 2C \|\eta - \eta_N\|_{L^2}+C  |r| \|\eta_N\|_{L^4}^2.
 \end{align*}
 \end{small}
 Take now $\varepsilon>0$ and choose first $N\geq 1$ such that 
 $2C \|\eta - \eta_N\|_{L^2} < \varepsilon/2
 $
 and then $\delta>0$ such that  $C \delta \|\eta_N\|_{L^4}^2 < \varepsilon/2$.
Then, we get $I_t  (|r|,\eta)<\varepsilon$ for every $|r|<\delta$. By arbitrariness of $\varepsilon$, we have proved \eqref{eq:limit_directional_derivative}.\\
\emph{Step 3.} We now prove  (ii).  Clearly, the adjoint operator $\left (D_{\xi}X_t^{\xi} \right )^* \in \mathcal L  (L^2(\mathcal F_t;\mathbb R^d); L^2(\mathcal F_0;\mathbb R^d) )$ satisfies  
\begin{align*}
&\left (D_{\xi}X_t^{\xi} \right)^*=\left( \partial_{x_0} X_t^{\xi, \P_{\xi}} \right)^* +\left( \partial_{\xi} X_{t}^{\xi, \P_{\xi}}  \right)^*.
\end{align*}
Notice that  $\left ( \nabla X_t^{\xi} \right)^* \in  \mathcal L\left (L^2(\mathcal F_t;\mathbb R^d); L^2(\mathcal F_0;\mathbb R^d)\right)$ satisfies 
\begin{small}
\begin{align*}
\left \langle \left ( \nabla X_t^{\xi} \right)^* \xi_t, \eta \right  \rangle_{L^2}&=\left \langle \xi_t,  \nabla X_t^{\xi}  \eta \right  \rangle_{L^2}= \E\left [\left\langle \xi_t,   \nabla X_t^{\xi}  \eta \right\rangle_{\R^{d}} \right]=\E\left [\left\langle\left(  \nabla X_t^{\xi} \right)^T \xi_t ,   \eta\right\rangle_{\R^{d}} \right] = \left \langle \left ( \nabla X_t^{\xi} \right)^T \xi_t, \eta \right  \rangle_{L^2},
\end{align*}
\end{small}
for every $\xi_t \in L^2(\mathcal F_t;\mathbb R^d), \eta \in L^2(\mathcal F_0;\mathbb R^d)$. Hence, $\left (D_{\xi}X_t^{\xi} \right )^*$ is characterized by 
\begin{equation}\label{eq:D_X0_XtX0*}
\left[ \left (D_{\xi}X_t^{\xi} \right)^*\eta \right](\omega)= \left (\partial_{x_0} X_t^{\xi, \P_{\xi}} \right)^T(\omega) \eta(\omega) +\left[ \left(\partial_{\xi} X_{t}^{\xi, \P_{\xi}} \right)^* \eta\right](\omega) .
\end{equation}

As for the estimate on the operator norm of $\left (D_{\xi} X_{t}^{\xi}  \right)^*$, it directly follows  from the corresponding estimate for  $D_{\xi} X_{t}^{\xi} $.

Finally, we prove the uniform strong continuity estimate of  $\left (D_{\xi} X_{t}^{\xi}  \right)^*$. Indeed, we have
\begin{equation*}
\begin{aligned}
\left \| \left (D_{\xi} X_{t}^{\xi} - D_{\xi} X_{t}^{\xi'} \right)^* \eta \right \|_{L^2}  & \leq \left \| \left ( \nabla X_t^{\xi, \P_{\xi}} -  \nabla X_t^{\xi', \P_{\xi'}} \right)^T \eta  \right \|_{L^2}  \\
& \quad  +\left \| \left ( \partial_{\xi} X_t^{\xi, \P_{\xi}}  - \partial_{\xi} X_t^{\xi', \P_{\xi'}} \right)^* \eta\right \|_{L^2}  =: \tilde E_t^1(\eta)+\tilde E_t^2 (\eta).
\end{aligned}
\end{equation*}
First, by \eqref{eq:norm_estimate_DX0_XtX0-DX0_XtX0prime} we have
\begin{align*}
\tilde E_t^2(\eta)&\leq \left \|\left(  \partial_{\xi} X_t^{\xi, \P_{\xi}}  - \partial_{\xi} X_t^{\xi', \P_{\xi'}} \right)^*\right \|_{\mathcal L
} \| \xi-\xi' \|_{L^2} \leq C  \| \xi-\xi' \|_{L^2}.
\end{align*}
Next, for every  $\eta \in L^\infty((\Omega,\mathcal F_t,\P);\mathbb R^d)$, by \eqref{eq:partial_x0XtX0-partial_x0XtX0'} we have 
\begin{align*}
&\tilde E^1_t(\eta) \leq \left \| \left( \nabla X_t^{\xi'} -  \nabla X_t^{\xi}\right)^T\right \|_{L^2} \left \|  \eta \right \|_{L^\infty} \leq   C\|\xi'-\xi \|_{L^2} \left \|  \eta \right \|_{L^\infty} .
\end{align*}
Hence, for $\eta \in L^2(\mathcal F_t;\mathbb R^d)$, we can proceed as we did with  $E^1_t$ in the previous step to get $\tilde E^1_t(\eta)\leq \varpi_{\eta} (\|\xi'-\xi \|_{L^2})$. The claim follows.
\end{proof}
\begin{remark}\label{rem:comparison_buck}{\rm (i)} We report here an {argument} communicated to us by R. Buckdahn and J. Li, which shows that  in general we may not expect that the map $
L^2(\mathcal F_0;\mathbb R^d) \to L^2(\mathcal F_t;\mathbb R^d)$, $\xi \mapsto X_t^{ \xi} 
$ is Fréchet differentiable. Consider \eqref{eq:sde} with $b^0(x,\mu)=b^0(x)$, $b^1(x,\mu)=b^1(x)$ and $b^0,b^1 \in C^2$ with  bounded and Lipschitz first and second order derivatives.
Fix $t \in [0,T]$ and let $\xi,\eta  \in L^2(\mathcal F_0;\mathbb R^d)$. Then, with usual notations, we have $X_t^{\xi}=X_t^{x}\mid_{x=\xi}$, $X_t^{\xi+\eta}-X_t^{\xi}=\int_0^1 \nabla  X_t^{\xi+\theta \eta} \eta  d\theta$, and
\begin{align*}
X_t^{\xi+\eta}-X_t^{\xi}-  \nabla X_t^{\xi} \eta =\int_0^1 \left [  \nabla  X_t^{\xi+\theta \eta}    - \nabla X_t^{\xi}\right] d\theta \ \eta=\left [\int_0^1\int_0^\theta  \partial_{x_0^2} X_t^{\xi+\lambda \eta}  d\lambda d\theta \ \eta \right]   \eta,
\end{align*} 
where the second order derivative  process  $\partial_{x_0^2} X_t^{\xi}$ was defined in \cite[Section 5]{buckdahn_li_peng_rainer}. Hence, in general we may only expect
\begin{small}
\begin{align}\label{eq:DX0Xt_not_expected_to_be-frechet}
\left \| X_t^{\xi+\eta}-X_t^{\xi}-  \nabla X_t^{\xi} \eta \right \|_{L^2} =\left( \E \left \{ \left | \left [ \int_0^1\int_0^\theta  \partial_{x_0^2} X_t^{\xi+\lambda \eta}  d\lambda d\theta \ \eta \right] \eta \right |^2 \right\} \right)^{1/2} \approx  \left ( \E \left [|\eta|^4 \right] \right)^{1/2},
\end{align} 
\end{small}
and not a behavior of the type $o\left ( \left ( \E \left [|\eta|^2 \right] \right)^{1/2} \right)$ for $ \left ( \E \left [|\eta|^2 \right] \right)^{1/2} \to 0,$ as required by the definition of Fréchet differentiability.
\end{remark}

\section{DRO Sensitivity for a systemic risk model}
\label{sec:example}
We consider an extension of the systemic risk model presented in  \cite[Section 2]{carmona_fouque_sun} to the case of multiplicative noise. Consider a model of interbank borrowing and lending of $N$ banks, where the log-monetary reserve of each bank $i $ given by the following $1$-dimensional SDE 
\begin{align} \d X^i_s &= \; \frac {a} {N} \sum_{j=1}^N (X_t^j-X_t^i) \,\d s + \sigma(X_t^i)\,\d B^i_s,\qquad X^i_t = \xi^i,\end{align} 
where $a >0$ measures the rate of borrowing/lending between bank $i$ and bank $j$, $\sigma \in C^{1,1}_b(\mathbb R)$ is the volatility coefficient of the bank reserve, $B^i$ are standard real-valued i.i.d. Wiener processes, and $\xi^i\sim \mu_0\in \mathcal P_2(\mathbb R)$ are i.i.d random variables.
When $N\to\infty$, the log-monetary reserve of a representative bank is provided by the McKean-Vlasov SDE
\begin{align} 
\d X_t &= \; 
a\big(
\E[X_t]-X_t  \big)  \,\d t + \sigma(X_t)\,\d B_t,\qquad X_0 = \xi\in L^2(\mathcal F_0), \ \P_\xi = \mu_0\in \mathcal P_2(\mathbb R),
\end{align} 
which is covered by Example \ref{ex:example_coeff}, with $$b(x,\mu):=\left (a\bigg(\int_{\mathbb R} y\mu(dy)-x\bigg),\sigma(x) \right ),  \ \ \ \ \partial_x b (x,\mu) = (-a,\partial_x\sigma(x)), \ \ \ \partial_{\tilde x} \partial_{\mu} b (x,\mu,\tilde x) = (a,0).$$ Moreover, we have
\begin{equation}\label{eq:expectation_log_monetary}
    \E \left [X_t^{\xi}\right]=\E[\xi], \quad t \geq 0.
\end{equation}

In order to evaluate systemic risk, according to the criterion in \cite{carmona_fouque_sun}, we evaluate the variance of the log monetary reserve at some time $T>0$, i.e. 
$$\phi(\mu_T)=\mbox{Var}[X_T^\xi]=\E[(X_T^\xi)^2]-(\E[X_T^\xi])^2,$$
where  
$$\phi \colon  \mathcal{P}_2(\mathbb{R}) \rightarrow \mathbb R, \ \  \phi(\mu):=\int_{\mathbb R}x^2\mu(dx)-\left(\int_{\mathbb R}x\mu(dx)\right)^2.$$
Notice that both the functional $\int_{\mathbb R}x^2\mu(dx)$ and  $\left(\int_{\mathbb R}x\mu(dx)\right)^2$  satisfy our  assumptions (see Examples \ref{ex:linear_functionals} and \ref{ex:functions_of_linear_functionals}) and we have  $\partial_x \delta_\mu \phi (\mu,x )=2x-2\int_{\mathbb R} y \mu(dy)$. Hence, with the notations of Section \ref{sec:sensitivity}, we may consider
$$
\Phi(\mu_0 ,\delta) := \sup_{\mu_0' \in B^2_\delta(\mu_0)} \phi \big(\mu_T'\big), \ \ \   \quad    \delta \geq 0,
$$
and apply Theorem \ref{th:derivative_varphi} to evaluate the sensitivity with respect to the initial distribution, measuring  the change in the variance of the log-monetary reserve with respect to the initial distribution. We get 
\begin{align}\label{eq:sensitivity_systemic_risk}
\frac{ \partial {\Phi}}{\partial \delta} (\mu_0,0 ) =2 \  \Bigg\|  \big(D_{\xi}X_T^{\xi} \big)^* \left( X_T^{\xi}-\E \left [X_T^{\xi}\right] \right) \Bigg\|_{L^2}=2 \  \Bigg\|  \big(D_{\xi}X_T^{\xi} \big)^* \left( X_T^{\xi}-\E \left [\xi\right] \right) \Bigg\|_{L^2},
\end{align}
where we have used \eqref{eq:expectation_log_monetary} and where $D_{\xi} X_t^{\xi} \in \mathcal L\left (L^2(\mathcal F_0); L^2(\mathcal F_t)\right)$, $t \in [0,T]$, is provided by Theorem \ref{th:derivative_XtX0_wrt_X0} and  \eqref{deff}. Here,   the latter is characterized as solution to 
\begin{align*}
D_{\xi}X_t^{\xi}\eta&= \eta+\int_0^t  \left(-aD_{\xi}X_s^{\xi}\eta +a  \E\left[  D_{\xi} X_s^{ \xi}  \eta \right]\right) ds+\int_0^t  \partial_x \sigma(X_s^{\xi})D_{\xi}X_s^{\xi}\eta \ dB_s \\
&=\eta+\int_0^t  \left(-aD_{\xi}X_s^{\xi}\eta +a  \E\left[  \eta \right]\right) ds +\int_0^t  \partial_x \sigma(X_s^{\xi})D_{\xi}X_s^{\xi}\eta \ dB_s, & \eta \in L^2(\mathcal F_0),
\end{align*}
where we used that  $\E\left [D_{\xi}X_t^{\xi}\eta \right]=\E[\eta]$.

\begin{flushleft}
\small{\textbf{Acknowledgments.}   The authors are very grateful to R. Buckdahn and J. Li for the content of Remark \ref{rem:comparison_buck}.
In particular, Filippo de Feo is very grateful to R. Buckdahn for valuable discussions concerning the  differentiability  with respect to the initial datum of the solution of the  McKean-Vlasov SDE and for the hospitality at Universite de Bretagne Occidentale, Brest.\medskip\\
\textbf{Funding.} Filippo de Feo and Fausto Gozzi were supported by the Italian Ministry of University and Research (MIUR), in the framework of PRIN projects 2017FKHBA8 001 (The Time-Space Evolution of Economic Activities: Mathematical Models and Empirical Applications) and 20223PNJ8K (Impact of the Human Activities on the Environment and Economic Decision Making in a Heterogeneous Setting: Mathematical Models and Policy Implications). Filippo de Feo acknowledges support from DFG CRC/TRR 388 "Rough
Analysis, Stochastic Dynamics and Related Fields", Project B05 and by  INdAM (Instituto Nazionale di Alta Matematica F. Severi) - GNAMPA (Gruppo Nazionale per l'Analisi Matematica, la Probabilità e le loro Applicazioni).  Salvatore Federico is a member of the GNAMPA project ``Problemi di controllo ottimo stocastico con memoria ed informazione parziale'' funded by INdAM and of the PRIN project 2022BEMMLZ ``Stochastic control and games and the role of information'' - CUP: D53D23005780006 funded by the Italian Ministry of University and Research (MUR).}
\end{flushleft}

\bibliographystyle{amsplain}

\end{document}